\pdfoutput=1 %
\documentclass[12pt]{article}
\usepackage[a4paper, total={6in, 9in}]{geometry}
\usepackage[utf8]{inputenc}
\usepackage{amssymb,amsmath}
\usepackage{mathtools}
\usepackage{bm}
\usepackage{bbm}
\usepackage{tensor}
\usepackage{xfrac}
\usepackage{amsthm}
\usepackage{enumerate}
\usepackage{xcolor}
\usepackage{diffcoeff}
\usepackage{array}
\usepackage{tikz}
\usetikzlibrary{cd}
\usepackage[style=alphabetic, maxalphanames=5, maxbibnames=99, url=false]{biblatex}
\renewbibmacro{in:}{}
\AtEveryBibitem{%
  \clearfield{doi}%
  \clearfield{issn}%
  \clearfield{isbn}%
  \clearlist{language}%
  \clearfield{month}%
  \clearfield{day}%
  \clearfield{pubstate}%
}
\usepackage{hyperref}

\theoremstyle{plain}
\newtheorem{theorem}{Theorem}[section]

\newtheorem{lemma}[theorem]{Lemma}
\newtheorem{corollary}[theorem]{Corollary}
\newtheorem{proposition}[theorem]{Proposition}
\newtheorem*{theorem*}{Theorem}
\newtheorem*{lemma*}{Lemma}
\newtheorem*{corollary*}{Corollary}
\newtheorem*{proposition*}{Proposition}

\theoremstyle{definition}
\newtheorem{definition}[theorem]{Definition}
\newtheorem*{definition*}{Definition}

\theoremstyle{remark}
\newtheorem{remark}[theorem]{Remark}

\theoremstyle{plain}
\newtheorem{introtheorem}{Theorem}

\newtheorem{introlemma}[introtheorem]{Lemma}

\newtheorem{introcorollary}[introtheorem]{Corollary}

\DeclarePairedDelimiter\ceil{\lceil}{\rceil}
\DeclarePairedDelimiter\floor{\lfloor}{\rfloor}

\DeclareFontFamily{T1}{cbgreek}{}
\DeclareFontShape{T1}{cbgreek}{m}{n}{<-6>  grmn0500 <6-7> grmn0600 <7-8> grmn0700 <8-9> grmn0800 <9-10> grmn0900 <10-12> grmn1000 <12-17> grmn1200 <17-> grmn1728}{}
\DeclareSymbolFont{quadratics}{T1}{cbgreek}{m}{n}
\DeclareMathSymbol{\qoppa}{\mathord}{quadratics}{19}
\DeclareMathSymbol{\Qoppa}{\mathord}{quadratics}{21}

\title{On the geometric fixed points of the real topological cyclic
  homology of $\mathbb{Z}/4$}
\author{Thomas Read\\\texttt{thjread@gmail.com}}
\date{}

\begin{document}

\DeclarePairedDelimiter\abs{\lvert}{\rvert}%
\DeclarePairedDelimiter\norm{\lVert}{\rVert}%

\makeatletter
\let\oldabs\abs
\renewcommand\abs{\@ifstar{\oldabs}{\oldabs*}}
\let\oldnorm\norm
\renewcommand\norm{\@ifstar{\oldnorm}{\oldnorm*}}
\makeatother

\maketitle

\section*{Abstract}

We study the homotopy groups of the geometric fixed points of the real topological cyclic homology of
$\mathbb{Z}/4$. %
We relate these groups to the values of the non-abelian derived functors of
the functor $M \mapsto (M \otimes_{\mathbb{Z}/4} M)^{C_2}$ at the
$\mathbb{Z}/4$-module $\mathbb{Z}/2$, %
which we precisely calculate with computer assistance up to degree $6$,
and calculate in general up to slight remaining ambiguity. Using these results
we compute $\pi_i(\mathrm{TCR}(\mathbb{Z}/4)^{\phi \mathbb{Z}/2})$ exactly
for $i \le 1$, up to an extension problem for $2
\le i \le 5$, and describe the asymptotic growth of this group for large $i$.
A consequence of these computations is that there exists some $0 \le i \le 5$ such
that the canonical map comparing the
genuine symmetric and symmetric $L$-theory spectra of $\mathbb{Z}/4$ is not an
isomorphism on degree $i$ homotopy,
and moreover this comparison map is never an isomorphism on homotopy in sufficiently large degrees.

\tableofcontents

\section*{Introduction}
\addcontentsline{toc}{section}{Introduction}

Trace methods, introduced in \cite{bokstedt_cyclotomic_1993}, study the $K$-theory
spectrum of a ring spectrum $A$ by mapping
it to other invariants that are more tractable to compute. The most basic such
invariant is the topological Hochschild homology spectrum $\mathrm{THH}(A)$,
from which one can obtain further invariants such as the $p$-typical topological
cyclic homology spectrum $\mathrm{TC}(A; p)$. If $A$ is a
ring spectrum with anti-involution, $\mathrm{THH}(A)$ refines to a real cyclotomic spectrum
$\mathrm{THR}(A)$, the real topological Hochschild homology (introduced in \cite{hesselholt_real_2015}, \cite{dotto_stable_2012} and \cite{hogenhaven_real_2016}). This receives a
trace map from the real algebraic $K$-theory $\mathrm{KR}(A)$ of \cite{hesselholt_real_2015}.
From $\mathrm{THR}(A)$ one can obtain further approximations to real $K$-theory,
in particular the $p$-typical real topological cyclic
homology, a genuine $\mathbb{Z}/2$-spectrum denoted $\mathrm{TCR}(A; p)$ (defined in
\cite{hogenhaven_real_2016}). While calculations of these objects exist for
various special classes of rings, there remain very simple rings for which
calculation is rather hard. For example the celebrated recent work of Antieau, Krause and
Nikolaus \cite{antieau_k-theory_2024} gives an algorithm for computing
the topological cyclic homology $\mathrm{TC}(\mathbb{Z}/p^n; p)$; before this even $\mathrm{TC}(\mathbb{Z}/4; 2)$
was not well understood. In this paper we will study the geometric fixed points
of the real topological cyclic homology $\mathrm{TCR}(\mathbb{Z}/4; 2)$.

Dotto, Moi and Patchkoria analyse the dihedral
structure of $\mathrm{THR}(A)$ in \cite{dotto_geometric_2024}, and use this to obtain a
formula for the geometric fixed points of $\mathrm{TCR}(A; p)$.
In particular when
$p=2$ they show that for $A$ a connective ring spectrum with anti-involution,
the spectrum $\mathrm{TCR}(A; 2)^{\phi \mathbb{Z}/2}$ can be expressed as a
homotopy equaliser of spectra
\begin{equation}\label{eq:htpy_eq}\mathrm{TCR}(A; 2)^{\phi \mathbb{Z}/2} \simeq \mathrm{hoeq}\left( (\mathrm{THR}(A)^{\phi \mathbb{Z}/2})^{C_2} \, \substack{\xrightarrow{f}\\[-0.2em] \xrightarrow[r]{}} \, \mathrm{THR}(A)^{\phi
  \mathbb{Z}/2} \right) \end{equation}
where $C_2$ denotes the Weyl group of $\mathbb{Z}/2$ in $O(2)$.
We can consider $A^{\phi \mathbb{Z}/2}$ as an $A$-module spectrum\footnote{Strictly
speaking, we should say a module over the underlying ring spectrum of the ring spectrum with
anti-involution $A$.} via the Frobenius
module structure
\[A \xrightarrow{\Delta} (N^{\mathbb{Z}/2}_{\{e\}}(A))^{\phi
  \mathbb{Z}/2} \xrightarrow{\mu^{\phi \mathbb{Z}/2}} A^{\phi \mathbb{Z}/2}\]
(where $N^{\mathbb{Z}/2}_{\{e\}}({-})$ is the Hill-Hopkins-Ravenel norm of
\cite{hill_nonexistence_2016-1}, see Section~\ref{sec:pointset} for more
details).
Then there is an equivalence of $C_2$-spectra
\[\mathrm{THR}(A)^{\phi \mathbb{Z}/2} \simeq A \otimes_{N_{\{e\}}^{C_2}A}
  N^{C_2}_{\{e\}}(A^{\phi \mathbb{Z}/2}) \text{,}\]
(where $A$ is a $N_{\{e\}}^{C_2}A$-spectrum via the multiplication map $\mu :
N_{\{e\}}^{C_2}A \to A$ and $N^{C_2}_{\{e\}}(A^{\phi \mathbb{Z}/2})$ is a
$N_{\{e\}}^{C_2}A$-spectrum via applying the norm to the Frobenius module
structure on $A^{\phi \mathbb{Z}/2}$).
If
$A^{\phi \mathbb{Z}/2}$ splits as an $A$-module into the direct sum of its
homotopy groups
\[A^{\phi \mathbb{Z}/2} \simeq \bigoplus_{n \ge 0} \Sigma^nH(\pi_n A^{\phi \mathbb{Z}/2})\]
then we obtain corresponding decompositions of the terms
$(\mathrm{THR}(A)^{\phi \mathbb{Z}/2})^{C_2}$ and $\mathrm{THR}(A)^{\phi
  \mathbb{Z}/2}$ appearing in the equaliser (\ref{eq:htpy_eq}), and \cite{dotto_geometric_2024} identifies the maps $f$ and
$r$ under these decompositions.
This approach is a key ingredient in the identification of the homotopy type of
$\mathrm{TCR}(\mathbb{Z}; 2)^{\phi \mathbb{Z}/2}$ (\cite{dotto_geometric_2024})
and the computation of the homotopy groups of $\mathrm{TCR}(k; 2)^{\phi
  \mathbb{Z}/2}$ for $k$ a field (see \cite{dotto_geometric_2024} for fields of odd
characteristic and perfect fields of characteristic $2$, extended by Dotto in
\cite{dotto_analogue_2025} to all fields).

The purpose of the present paper is to apply this decomposition to study $\mathrm{TCR}(\mathbb{Z}/4; 2)^{\phi
  \mathbb{Z}/2}$. From this point onwards we will write $\mathrm{TCR}(A)$ for
$\mathrm{TCR}(A; 2)$, leaving the prime $p = 2$ implicit. In Section~\ref{sec:decomposition} we show that
$(H\underline{\mathbb{Z}/4})^{\phi \mathbb{Z}/2}$ (with the Frobenius module
structure, and where the underline indicates we are considering $\mathbb{Z}/4$ as
having trivial involution) splits as an
$H\underline{\mathbb{Z}/4}$-module into the direct sum of its homotopy groups,
and so following the above approach
we obtain a decomposition of the equaliser diagram (\ref{eq:htpy_eq}) for $A =
H\underline{\mathbb{Z}/4}$. The homotopy groups of
$\mathrm{THR}(\mathbb{Z}/4)^{\phi \mathbb{Z}/2}$ are immediate from the
decomposition, while to calculate the homotopy groups of
$(\mathrm{THR}(\mathbb{Z}/4)^{\phi \mathbb{Z}/2})^{C_2}$ we need to analyse
summands of the form
\[\left( H \underline{\mathbb{Z}/4} \otimes_{N_{\{e\}}^{C_2}(H\mathbb{Z}/4)}
    N^{C_2}_{\{e\}}(\Sigma^n H\mathbb{Z}/2) \right)^{C_2}\]
for $n \in \mathbb{N}$. In Section~\ref{sec:computation_of_homotopy} we show that the homotopy groups of these summands can
be expressed using non-abelian derived functors.
Non-abelian
derived functors were introduced by Dold and Puppe in
\cite{dold_non-additive_1958} and developed in \cite{dold_homologie_1961}. Just
as derived functors are computed using free resolutions in chain complexes,
non-abelian derived functors are computed using free resolutions in simplicial modules; we
later recall this in Definition~\ref{def:non-ab_derived}. The key result is
the following (proved in Corollary~\ref{cor:F_non_ab}):

\begin{introlemma} \label{lem:intro_nonab}
  Let $F_{R} : \mathrm{Mod}_{R} \to
\mathrm{Ab}$ defined by $F_{R}(M) = (M \otimes_{R} M)^{C_2}$.
  Then
\[\pi_i\left(\left( H \underline{\mathbb{Z}/4} \otimes_{N_{\{e\}}^{C_2}(H\mathbb{Z}/4)}
    N^{C_2}_{\{e\}}(\Sigma^n H\mathbb{Z}/2) \right)^{C_2}\right) \cong L_i^{(n)}F_{\mathbb{Z}/4}(\mathbb{Z}/2)\]
where $L_i^{(n)}F_{R} : \mathrm{Mod}_{R} \to \text{Ab}$ is the $i$th non-abelian derived functor of type $n$
of $F_{R}$.
\end{introlemma}

To compute these values we initially turn to brute force calculation,
using computer assistance to evaluate $L_i^{(n)}F_{\mathbb{Z}/4}(\mathbb{Z}/2)$
for $i, n \le 6$.

\begin{introtheorem}\label{thm:intro_thr_groups}
The low degree homotopy groups of the terms in the equaliser (\ref{eq:htpy_eq})
for $A = H\underline{\mathbb{Z}/4}$ are as follows:
\[\begin{tabular}{>{$}c<{$}>{$}c<{$}>{$}c<{$}>{$}c<{$}>{$}c<{$}}
  i & \pi_i((\mathrm{THR}(\mathbb{Z}/4)^{\phi \mathbb{Z}/2})^{C_2}) & \pi_i(\mathrm{THR}(\mathbb{Z}/4)^{\phi \mathbb{Z}/2}) \\ \hline
  0 & \mathbb{Z}/4 & \mathbb{Z}/2 \\
  1 & (\mathbb{Z}/2)^4 & (\mathbb{Z}/2)^3 \\
  2 & (\mathbb{Z}/2)^7 & (\mathbb{Z}/2)^6 \\
  3 & (\mathbb{Z}/2)^{13} & (\mathbb{Z}/2)^{10} \\
  4 & (\mathbb{Z}/2)^{15} \oplus (\mathbb{Z}/4)^3 & (\mathbb{Z}/2)^{15} \\
  5 & (\mathbb{Z}/2)^{27} & (\mathbb{Z}/2)^{21}\\
  6 & (\mathbb{Z}/2)^{34} & (\mathbb{Z}/2)^{28}
\end{tabular}\]
\end{introtheorem}

In Section~\ref{sec:action_on_htpy} we continue by computing the action on homotopy of the above-mentioned maps $f$ and $r$
up to degree $6$. In the case of $f$ this is a straightforward extension of
our computer program; for the map $r$ we do a more involved analysis.
As a result we can write down the long exact sequence associated to the homotopy
equaliser of $f$ and $r$, and so obtain short exact sequences describing
$\pi_i(\mathrm{TCR}(\mathbb{Z}/4)^{\phi \mathbb{Z}/2})$ for $-1 \le i \le 5$. That
is, we compute the desired homotopy groups up to solving certain extension problems.
For $i=-1$ and $i=0$ these extension problems are trivial, and in
Section~\ref{sec:extension_i_1} we solve the extension problem for $i=1$. Our results are summarised in the
following:

\begin{introtheorem} \label{thm:intro_tcr_groups}
  All non-trivial homotopy groups of $\mathrm{TCR}(\mathbb{Z}/4)^{\phi \mathbb{Z}/2}$ are
  in degrees $\ge -1$.
  We compute
  \begin{align*}\pi_{-1}(\mathrm{TCR}(\mathbb{Z}/4)^{\phi \mathbb{Z}/2}) &\cong \mathbb{Z}/2\\
  \pi_{0}(\mathrm{TCR}(\mathbb{Z}/4)^{\phi \mathbb{Z}/2}) &\cong \mathbb{Z}/4\\
  \pi_{1}(\mathrm{TCR}(\mathbb{Z}/4)^{\phi \mathbb{Z}/2}) &\cong
    (\mathbb{Z}/2)^2 \mathrm{,}\end{align*}
  and we have short exact sequences
\begin{align*}
  0 \to \mathbb{Z}/2 \to &\pi_{2}(\mathrm{TCR}(\mathbb{Z}/4)^{\phi \mathbb{Z}/2}) \to (\mathbb{Z}/2)^2 \to 0\\
  0 \to (\mathbb{Z}/2)^2 \to &\pi_{3}(\mathrm{TCR}(\mathbb{Z}/4)^{\phi \mathbb{Z}/2}) \to (\mathbb{Z}/2)^4 \to 0\\
  0 \to (\mathbb{Z}/2)^4 \to &\pi_{4}(\mathrm{TCR}(\mathbb{Z}/4)^{\phi \mathbb{Z}/2}) \to (\mathbb{Z}/2)^6 \oplus \mathbb{Z}/4 \to 0\\
  0 \to (\mathbb{Z}/2)^9 \to &\pi_{5}(\mathrm{TCR}(\mathbb{Z}/4)^{\phi \mathbb{Z}/2}) \to (\mathbb{Z}/2)^{10} \to 0 \mathrm{.}
\end{align*}
Moreover the groups $\pi_i(\mathrm{TCR}(\mathbb{Z}/4)^{\phi \mathbb{Z}/2})$ are
$4$-torsion for all $i$.
\end{introtheorem}

In particular note $\pi_1$ and $\pi_5$ cannot be isomorphic.
Demonstrating this was part of our
original motivation for computing these homotopy groups. As we explain in
the following subsection, via showing that the homotopy groups of
$\mathrm{TCR}(\mathbb{Z}/4)^{\phi \mathbb{Z}/2}$ are not $4$-periodic in
positive degrees, we can conclude that the homotopy groups of the genuine symmetric
and symmetric $L$-theories of $\mathbb{Z}/4$ are not canonically isomorphic in
non-negative degrees, resolving a question of Calm\`es et al.\ from 
\cite{calmes_hermitian_2021} Example~1.3.11.

In Section~\ref{sec:nadf} we return to a more theoretical approach, studying the
non-abelian derived functors $L^{(n)}_iF_{\mathbb{Z}/4}(\mathbb{Z}/2)$ for
general $i$ and $n$. The non-abelian derived functors of
$F_{R}$ are isomorphic to those of J. H. C. Whitehead's functor
$\Lambda_R$
of \cite{whitehead_certain_1950} (since $F_R$ and $\Lambda_R$ agree on free modules); the non-abelian
derived functors of $\Lambda_R$ were previously studied in \cite{simson_connected_1974} and
\cite{simson_stable_1974} %
but most of their results are only valid when $2$ is not a zero-divisor in $R$. We almost precisely narrow down the
value of $L^{(n)}_iF_{\mathbb{Z}/4}(\mathbb{Z}/2)$ for all $i$ and $n$, with just slight
uncertainties remaining.

\begin{introtheorem} \label{thm:intro_thm_nadf}
  For $i \ge 0$ fixed, $L_i^{(n)}F_{\mathbb{Z}/4}(\mathbb{Z}/2)$ takes the same value
  for all $0 \le n \le \floor{i/2}$. Moreover we determine this constant value, up to two possibilities in some cases:
  \begin{align*}L_i^{(0)} F_{\mathbb{Z}/4}(\mathbb{Z}/2) &\cong L_i^{(1)}
    F_{\mathbb{Z}/4}(\mathbb{Z}/2) \cong \dotsb \cong L_i^{\floor{i/2}}
    F_{\mathbb{Z}/4}(\mathbb{Z}/2)\\
    &\cong \begin{cases}
      (\mathbb{Z}/2)^{\floor{i/2}+1} \text{ or } (\mathbb{Z}/2)^{\floor{i/2}}
      \oplus \mathbb{Z}/4 \quad &\text{if $i \equiv 0 \text{ mod } 4$,}\\
      (\mathbb{Z}/2)^{\floor{i/2}+1} \text{ or } (\mathbb{Z}/2)^{\floor{i/2}+2} \quad &\text{if $i \equiv 1 \text{ mod } 4$,}\\
      (\mathbb{Z}/2)^{\floor{i/2}+1} \quad &\text{if $i \equiv 2 \text{ mod } 4$,}\\
      (\mathbb{Z}/2)^{\floor{i/2}+2} \quad &\text{if $i \equiv 3 \text{ mod } 4$.}
      \end{cases}\end{align*}
    For each $k$ we have
    $\abs{L_{4k}^{(0)}F_{\mathbb{Z}/4}(\mathbb{Z}/2)} =
    \abs{L_{4k+1}^{(0)}F_{\mathbb{Z}/4}(\mathbb{Z}/2)}$. That is, if we know
    the value for some $i$ congruent to $0$ mod $4$ then this tells us the value for $i+1$.

  For $\floor{i/2} < n \le i$ we have
  \[L_i^{(n)}F_{\mathbb{Z}/4}(\mathbb{Z}/2) \cong (\mathbb{Z}/2)^{i-n+1} \text{.}\]
  And $L_i^{(n)}F_{\mathbb{Z}/4}(\mathbb{Z}/2) = 0$ for $n > i$.
\end{introtheorem}

This result, proved in Theorem~\ref{thm:l_f_z2}, is enough to put quite tight bounds on the order of
$\pi_i((\mathrm{THR}(\mathbb{Z}/4)^{\phi \mathbb{Z}/2})^{C_2})$.\footnote{Note
  however that in Theorem~\ref{thm:intro_thm_nadf} we do not analyse the $f$ and $r$ maps from (\ref{eq:htpy_eq});
  this together with the remaining uncertainties in some degrees
  mean that this theorem is not sufficient
  to recover Theorem~\ref{thm:intro_tcr_groups}.} In
Section~\ref{sec:size_est} we use the long exact sequence in homotopy for the equaliser formula (\ref{eq:htpy_eq})
to get upper and lower bounds for the
orders of the groups $\pi_i(\mathrm{TCR}(\mathbb{Z}/4)^{\phi \mathbb{Z}/2})$.
Despite Theorem~\ref{thm:intro_tcr_groups} showing that these groups are not $4$-periodic in all positive degrees, one
might still have
imagined that beyond a certain point they could start repeating
with some period. The following result (proved in Theorem~\ref{thm:asymptotics})
shows that this is not the case; rather, the homotopy groups
continue to grow in size.

\begin{introtheorem} \label{thm:intro_asymptotics}
  The logarithm of the order of the group
  \[\pi_i(\mathrm{TCR}(\mathbb{Z}/4)^{\phi \mathbb{Z}/2})\]
  exhibits asymptotically quadratic growth with respect to $i$.

  More precisely, for $i \ge 0$ we have upper and lower bounds
  \[2^{(i-1)(i+1)/8} \le \abs{\pi_i\left(\mathrm{TCR}(\mathbb{Z}/4)^{\phi
          \mathbb{Z}/2}\right)} \le 2^{(i+2)(9i+20)/8} \text{.}\]
\end{introtheorem}

\subsection*{Relation to \texorpdfstring{$L$}{L}-theory}
\addcontentsline{toc}{subsection}{Relation to \texorpdfstring{$L$}{L}-theory}

Our original motivation for studying the geometric fixed points of
$\mathrm{TCR}$ was due to connections to $L$-theory via a result of
Harpaz, Nikolaus and Shah \cite{harpaz_real_nodate}. In this subsection we
describe the $L$-theoretic consequences of our calculations.

Let us first recall the relevant flavours of $L$-theory. Let $R$ be a
commutative ring, possibly equipped with an involution. Under the perspective of
Calm\`{e}s et al.\ \cite{calmes_hermitian_2023-1} we can study Poincar\'e structures on the perfect
derived $\infty$-category $\mathcal{D}^p(R)$, and their associated $L$-theory spectra. In
particular we have the symmetric structure $\Qoppa^{s}_R$ and the
genuine symmetric structure $\Qoppa^{gs}_R$ (see \cite{calmes_hermitian_2023}
4.2.12, 4.2.13, 4.2.23).
The $L$-theory spectrum (\cite{calmes_hermitian_2023-1} 4.4.4)
$L(\mathcal{D}^p(R), \Qoppa^{s}_R)$ associated to the symmetric
structure is Ranicki's $4$-periodic symmetric $L$-theory spectrum, while the spectrum $L(\mathcal{D}^p(R), \Qoppa^{gs}_R)$ associated to
the genuine symmetric structure has homotopy groups given by the non-periodic symmetric
$L$-groups defined by Ranicki in \cite{ranicki_algebraic_1992}.

We would like to understand the relation between these two different
$L$-theories. There is a canonical comparison map
\begin{equation}\label{eq:l-theory_comp} L(\mathcal{D}^p(R), \Qoppa^{gs}_R) \to L(\mathcal{D}^p(R), \Qoppa^{s}_R)\end{equation}
from the genuine symmetric $L$-theory to the symmetric $L$-theory. Theorem~6 of \cite{calmes_hermitian_2021}
says that if $R$ is a coherent ring with finite global dimension $d$, then this map is an injection on homotopy groups in degree $d-2$
and an isomorphism on homotopy in degrees $\ge d-1$. That is, for a sufficiently
nice ring $R$ these $L$-theories agree in sufficiently high degrees. In
\cite{calmes_hermitian_2021} Example~1.3.11 they show that
this bound is tight when allowing rings with involution: for the
$d$-dimensional ring $R = \mathbb{F}_2[\mathbb{Z}^d]$ with involution induced by the
inversion of the group $\mathbb{Z}^d$, the map (\ref{eq:l-theory_comp}) is not
surjective in degree $d-2$. However in general there are not many
rings $R$ where both $L$-theory spectra have been computed, and
after giving the above example Calm\`es et al.\ pose the question of identifying any commutative ring
with trivial involution such that it is not the case that the map (\ref{eq:l-theory_comp}) is
an isomorphism on all non-negative homotopy groups.

We use our results to answer this question, showing that in the
case of $R = \mathbb{Z}/4$ (considered as having trivial involution) the two
$L$-theories are far from always coinciding in non-negative degrees.

\begin{introcorollary}
  The comparison map from genuine symmetric to symmetric $L$-theories
  \[\pi_i\Big(L(\mathcal{D}^p(\mathbb{Z}/4), \Qoppa^{gs}_{\mathbb{Z}/4})\Big) \to
 \pi_i\Big(L(\mathcal{D}^p(\mathbb{Z}/4), \Qoppa^{s}_{\mathbb{Z}/4})\Big)
\]
fails to be an isomorphism for some $0 \le i \le 5$, and moreover for $i$ sufficiently
large it is never an isomorphism.
\end{introcorollary}
\begin{proof}
We use the result of
\cite{harpaz_real_nodate} relating $L$-theories to $\mathrm{TCR}$. Let $\Qoppa^{q}_R$
denote the quadratic Poincar\'e structure on $\mathcal{D}^p(R)$, whose
associated $L$-theory $L(\mathcal{D}^p(R), \Qoppa^{q}_R)$ is Ranicki's $4$-periodic
quadratic $L$-theory spectrum. The Poincar\'e structure $\Qoppa^{q}_R$ is the
initial Poincar\'e structure on $\mathcal{D}^p(R)$ compatible with the duality
$\mathrm{Hom}({-}, R)$ (\cite{calmes_hermitian_2023} 1.3.5), so there is a canonical map $L(\mathcal{D}^p(R),
\Qoppa^{q}_R) \to L(\mathcal{D}^p(R), \Qoppa^{gs}_R)$. The cofibre of this map
is known as normal $L$-theory. Harpaz, Nikolaus and Shah show that in fact normal
$L$-theory is equivalent to the geometric fixed points of $\mathrm{TCR}$, so we
have a cofibre sequence of spectra
\[L(\mathcal{D}^p(R), \Qoppa^{q}_R) \to L(\mathcal{D}^p(R), \Qoppa^{gs}_R) \to
  \mathrm{TCR}(R)^{\phi \mathbb{Z}/2} \mathrm{.}\]
Consider the map of cofibre sequences
\[\begin{tikzcd}
L(\mathcal{D}^p(R), \Qoppa^{q}_R) \ar[r] \ar[d, equal] & L(\mathcal{D}^p(R),
\Qoppa^{gs}_R) \ar[r] \ar[d] &
\mathrm{TCR}(\mathbb{Z}/4)^{\phi \mathbb{Z}/2} \ar[d]\\
L(\mathcal{D}^p(R), \Qoppa^{q}_R) \ar[r] & L(\mathcal{D}^p(R), \Qoppa^{s}_R)
\ar[r] & C \text{,}\end{tikzcd}
  \]
  where $C$ is the cofibre of $L(\mathcal{D}^p(R), \Qoppa^{q}_R) \to
  L(\mathcal{D}^p(R), \Qoppa^{s}_R)$.

Suppose the map $L(\mathcal{D}^p(R), \Qoppa^{gs}_R) \to
  L(\mathcal{D}^p(R), \Qoppa^{s}_R)$ is an isomorphism on homotopy groups
  $\pi_i$ for all $0 \le i \le 5$. Consider the map of long exact sequences in
  homotopy groups induced by the above map of cofibre sequences. By the $5$-lemma
  the map $\mathrm{TCR}(\mathbb{Z}/4)^{\phi \mathbb{Z}/2} \to C$ is an isomorphism on
  $\pi_i$ for $1 \le i \le 5$.
  Observe $L(\mathcal{D}^p(R), \Qoppa^{q}_R) \to
L(\mathcal{D}^p(R), \Qoppa^{s}_R)$ is a map of
$L(\mathcal{D}^p(R), \Qoppa^{s}_R)$-modules and $L(\mathcal{D}^p(R),
\Qoppa^{s}_R)$ is a $4$-periodic ring spectrum, so $C$ is $4$-periodic and hence we conclude that $\pi_1(\mathrm{TCR}(\mathbb{Z}/4)^{\phi \mathbb{Z}/2}) \cong
 \pi_5(\mathrm{TCR}(\mathbb{Z}/4)^{\phi \mathbb{Z}/2})$. However
 Theorem~\ref{thm:intro_tcr_groups}
 contradicts this, and so our earlier supposition must be false;
 that is, the comparison map between the homotopy groups of the genuine symmetric and
symmetric $L$-groups must fail to be an isomorphism somewhere in degrees $0$ to
$5$\footnote{In fact running the same argument more carefully shows that the map
fails to be an isomorphism for some $i \in \{0, 1, 4, 5\}$}.

Similarly suppose the map $L(\mathcal{D}^p(R), \Qoppa^{gs}_R) \to
  L(\mathcal{D}^p(R), \Qoppa^{s}_R)$ induces an injection on $\pi_i$ for some particular
  $i$. Then the $4$-lemma shows that
  $\pi_i(\mathrm{TCR}(\mathbb{Z}/4)^{\phi \mathbb{Z}/2}) \to \pi_i(C)$ is
  injective. But $C$ is $4$-periodic and the homotopy groups of
  $\mathrm{TCR}(\mathbb{Z}/4)^{\phi \mathbb{Z}/2}$ grow according to
  Theorem~\ref{thm:intro_asymptotics}, so for $i$ sufficiently large
  $\pi_i(\mathrm{TCR}(\mathbb{Z}/4)^{\phi \mathbb{Z}/2})$ cannot inject into
  $\pi_i(C)$. We conclude that above a certain degree the
  comparison map is never an injection (and so in particular never an isomorphism) on homotopy.
\end{proof}

\subsection*{Computer code and results}
\addcontentsline{toc}{subsection}{Computer code and results}

The computer program used to support the computations in this paper is available
at \url{https://github.com/thjread/tcr_geo_fixed}, together with the raw program output.
The computer code is written in Python, utilising the GAP computational group theory
library \cite{gap_group_gap_2024} via SageMath \cite{sage_developers_sagemath_2024}.

\subsection*{Acknowledgements}
\addcontentsline{toc}{subsection}{Acknowledgements}

I would like to thank Emanuele Dotto and Irakli Patchkoria for encouraging me
to study these homotopy groups, and for many helpful
conversations during the course of the project.

The author is supported by the Warwick Mathematics Institute Centre for Doctoral Training, and gratefully acknowledges funding from the University of Warwick.

\section{Preliminaries and conventions}\label{sec:conventions}

We first establish some notation, and review some technical details and
background material.

\subsection{Point-set models and model structures} \label{sec:pointset}

In a few places our argument depends on using a particular point-set model of
spectra. We will work with orthogonal $\mathbb{Z}/2$-equivariant spectra,
generally following Schwede's lecture notes \cite{schwede_lectures_2023} for
the basic constructions described as follows.

We denote the Eilenberg-MacLane spectrum of an abelian group $A$ by $HA$, and
model this as an orthogonal spectrum
via letting the $n$th level $(HA)_n$ be the reduced $A$-linearisation of the $n$-sphere $A[S^n]$.
Similarly we denote the Eilenberg-MacLane spectrum of an abelian group $A$ with involution
$w$ by $H(A, w)$, and model this as an orthogonal $\mathbb{Z}/2$-spectrum via
letting $H(A, w)_n = A[S^n]$ with the $\mathbb{Z}/2$-action induced by $w$.
Sometimes for $A$ an abelian group we want to work with the
$\mathbb{Z}/2$-spectrum $H(A, \text{id}_A)$, the Eilenberg-MacLane
spectrum of $A$ considered as having trivial involution; for convenience we still denote
this $HA$, and whether we mean a spectrum or a $\mathbb{Z}/2$-spectrum should be
clear from context.
We define the geometric fixed
points of a $\mathbb{Z}/2$-spectrum $X$ via $X^{\phi \mathbb{Z}/2}(M) \coloneqq X(M \otimes_{\mathbb{S}}
  \rho_{\mathbb{Z}/2})^{\mathbb{Z}/2}$, and we define the
  non-derived version of the norm of a spectrum $Y$ via $\tilde{N}_{\{e\}}^{\mathbb{Z}/2}(Y) \coloneqq Y
  \otimes_\mathbb{S} Y$ where $\mathbb{Z}/2$ acts by swapping the factors.

  We work with the flat
  model structure of \cite{stolz_equivariant_2011},
  \cite{brun_equivariant_2022}. This in particular means that given a strictly
  commutative ring spectrum we can choose a
  strictly commutative and flat cofibrant replacement.
  Except when explicitly stated we will always want to work with derived constructions on spectra. We let
  $N_{\{e\}}^{C_2}$ denote the derived Hill-Hopkins-Ravenel norm (i.e.\ taking a
  cofibrant replacement followed by applying $\tilde{N}_{\{e\}}^{C_2}$) and write
  $\otimes_R$ for the derived smash product relative to a ring spectrum $R$.
  Note that the above geometric fixed point construction does not need to be
  derived: it gives the correct homotopy type for any $\mathbb{Z}/2$-spectrum
  $X$ without needing to take a cofibrant replacement.

  For a $\mathbb{Z}/2$-spectrum $X$
we denote the homotopy groups of the genuine $C_2$-fixed points by
\[\pi_i^{C_2}(X) \coloneqq \pi_i\left( X^{C_2} \right) \text{.}\]

  We now describe the Frobenius module structure on the geometric fixed points of
  a commutative ring spectrum (\cite{dotto_real_2021} Section~2.5).
  There is a canonical map
  \[Y \xrightarrow{\Delta} (\tilde{N}_{\{e\}}^{C_2}(Y))^{\phi C_2} \mathrm{,}\]
  which is an equivalence of spectra whenever $Y$ is cofibrant. %
  If $R$ is a commutative ring spectrum then there is a canonical multiplication
  map \[\tilde{N}^{C_2}_{\{e\}}(R) \xrightarrow{\mu} R \mathrm{,}\]
    which is $C_2$-equivariant when we consider $R$ as a $C_2$-spectrum with
  trivial $C_2$-action. Using these maps we define the Frobenius module structure as follows.
  \begin{definition} \label{def:frobenius}
    The Frobenius $R$-module structure on $R^{\phi \mathbb{Z}/2}$ is defined by
    the map
    \[R \xrightarrow{\Delta} (N_{\{e\}}^{C_2}(R))^{\phi C_2}
      \xrightarrow{\mu^{\phi \mathbb{Z}/2}} R^{\phi \mathbb{Z}/2} \text{.}\]
  \end{definition}

  In fact although this definition uses the derived norm, we could have equally
  used the underived norm.

  \begin{lemma} \label{lem:frob_non_derived}
    The map defining the Frobenius module structure is homotopic to %
    \[R \xrightarrow{\Delta} (\tilde{N}_{\{e\}}^{C_2}(R))^{\phi C_2}
      \xrightarrow{\mu^{\phi \mathbb{Z}/2}} R^{\phi \mathbb{Z}/2} \text{.}\]
  \end{lemma}
  \begin{proof}
  Taking a commutative cofibrant replacement $R^{\text{cof}}$ for $R$, we have a commutative diagram
  \[\begin{tikzcd}
    R^{\text{cof}} \ar[r, "\Delta"] \ar[d, "\simeq"] &
    \left( \tilde{N}_{\{e\}}^{\mathbb{Z}/2}(R^\text{cof})
    \right)^{\phi \mathbb{Z}/2} \ar[r, "\mu^{\phi \mathbb{Z}/2}"] \ar[d] &
    (R^{\text{cof}})^{\phi \mathbb{Z}/2} \ar[d, "\simeq"]\\
    R \ar[r, "\Delta"] & \left(
      \tilde{N}_{\{e\}}^{\mathbb{Z}/2}R \right)^{\phi \mathbb{Z}/2}\ar[r,
    "\mu^{\phi \mathbb{Z}/2}"] & R^{\phi \mathbb{Z}/2} \text{.}
    \end{tikzcd}\]
  Now consider the two paths around the circumference from $R$ to $R^{\phi \mathbb{Z}/2}$.
  \end{proof}

  \subsection{Dold-Kan equivalence}

Let $R$ be a commutative ring. The Dold-Kan equivalence \cite{dold_homology_1958}
is an equivalence between the category of connective chain
complexes of $R$-modules $\mathrm{Ch}_{\ge 0}(R)$ and the category of simplicial
$R$-modules $\mathrm{sMod}_R$.
Let us briefly recall the definitions of these functors, since these
constructions will be an essential part of our calculations and we will need to
reason about the concrete details later on.

  We denote the functor from chain complexes to simplicial modules by
  \[\sigma : \mathrm{Ch}_{\ge 0}(R) \to \mathrm{sMod}_R \mathrm{.}\]
  Let $M_\ast$ be a connective chain complex. Then the simplicial $R$-module
  $\sigma M_\bullet$ is given level-wise by
  \[\sigma M_n = \bigoplus_{[n] \twoheadrightarrow [k]} M_k\]
  where the sum is taken over all surjections out of $[n]$ in the simplex
  category $\Delta$; that is, all surjective weakly increasing maps $[n]
  \twoheadrightarrow [k]$ for $0 \le k \le n$. The simplicial structure is as
  follows. Given a map $[m] \to [n]$ in $\Delta$, we need to produce a map
  $\sigma M_n \to \sigma M_m$. It suffices to describe how to map out of each
  direct summand of $\sigma M_n$, and so for each surjection $[n]
  \twoheadrightarrow [k]$ we will give a map $M_k \to \sigma M_m$. Factor the composite $[m] \to [n]
  \twoheadrightarrow [k]$ as a surjection followed by an injection $[m]
  \twoheadrightarrow [m'] \hookrightarrow [k]$, then we take the map $M_k \to
  M_{m'}$ from the chain complex structure of $M$, and compose it with the
  summand inclusion $M_{m'} \subseteq \sigma M_m$ corresponding to $[m]
  \twoheadrightarrow [m']$.

  The inverse functor from simplicial modules to chain complexes is the normalised
  chain complex functor
  \[N : \mathrm{sMod}_R \to \mathrm{Ch}_{\ge 0}(R) \mathrm{.}\]
  Given a simplicial $R$-module $M_\bullet$, we define $NM_\ast$ as follows: let
  $NA_n \coloneq \bigcap_{i=0}^{n-1} \text{ker}(d_i)$ be the subgroup of $A_n$ that is killed by all the face maps $d_i$ for
  $i < n$, and let the differential $NA_n \to NA_{n-1}$ be given by $(-1)^n d_n$.

\begin{theorem}[Dold-Kan equivalence
  \cite{dold_homology_1958}] \label{thm:dold_kan}%
  The functors $\sigma$ and $N$ give an equivalence of categories
  \[\mathrm{Ch}_{\ge 0}(R) \simeq \mathrm{sMod}_R \text{.}\]
\end{theorem}
\begin{proof}
  See \cite{goerss_simplicial_2009} Section III.2 (and note the correction to
  the proof in the errata).
\end{proof}

  Additionally recall the Moore complex of a simplicial $R$-module. Given a
  simplicial $R$-module $A_\bullet$, the Moore complex is the chain complex
  $A_\ast$ which is given by $A_n$ in degree $n$, and has boundary map $\partial : A_n
  \to A_{n-1}$ given by the alternating sum of the face maps $\Sigma_{i=0}^n
  (-1)^i d_i$. The Moore complex $A_\ast$ and the normalised chain complex
  $NA_\ast$ are naturally homotopy equivalent, and so in practice we can use
  either in situations where we only care about the result up to homotopy.

  \subsection{Real topological Hochschild homology}

  In this subsection we briefly recall the construction of $\mathrm{THR}(A)$ and $\mathrm{TCR}(A; p)$ in the
  special case where $A$ is a discrete commutative ring with trivial involution
  (this is the only case we will need). These constructions were introduced in
  \cite{hesselholt_real_2015}, \cite{dotto_stable_2012},
  \cite{hogenhaven_real_2016}; good sources for
  full details include \cite{dotto_real_2021} and \cite{dotto_geometric_2024}.

  Let $HA$ be a flat model for the Eilenberg-MacLane spectrum of $A$, in
  orthogonal spectra. The cyclic nerve of $HA$ becomes a dihedral spectrum if we
  give it an appropriate involution. We denote this construction $N^{di}HA$.
  In level $n$ we have $(N^{di}HA)_n = HA^{\otimes n+1}$, which we visualise as $n+1$ copies of $A$ tensored
  together in a circle, equipped with rotation and reflection actions. %
   We define the real topological Hochschild
  homology of $A$ to be the geometric realisation
  \[\mathrm{THR}(A) \coloneqq \abs{N^{di}A} \mathrm{.}\]
  This is an orthogonal spectrum with an $O(2)$-action, so we can regard it as a
  genuine $O(2)$-spectrum.
  Note the underlying cyclotomic spectrum of $\mathrm{THR}(A)$ is $\mathrm{THH}(A)$, while
  the underlying genuine
  $\mathbb{Z}/2$-spectrum is %
  \[A \otimes_{N_{e}^{\mathbb{Z}/2}A} A\]
  (where $A$ has $N_{e}^{\mathbb{Z}/2}A$-module structure given by
  $N_{e}^{\mathbb{Z}/2}A \xrightarrow{\mu} A$). In particular we find that the geometric fixed points are %
  \[\mathrm{THR}(A)^{\phi \mathbb{Z}/2} \simeq
    A^{\phi \mathbb{Z}/2} \otimes_A
    A^{\phi \mathbb{Z}/2} \mathrm{.}\]

  Let $p$ be a prime. We define the $p$-typical truncated real topological
  restriction homology of $A$ by
  \[\mathrm{TRR}^{n+1}(A; p) \coloneqq \mathrm{THR}(A)^{C_{p^{n}}} \mathrm{,}\]
  where $C_{p^n} \le D_{p^{n}} \le O(2)$. We can consider this as a
  $\mathbb{Z}/2$-spectrum, since the Weyl group of $C_{p^n}$ inside $D_{p^n}$ is
  $\mathbb{Z}/2$. There is a $\mathbb{Z}/2$-equivariant map $R : \mathrm{TRR}^{n+1}(A; p) \to
  \mathrm{TRR}^{n}(A; p)$ coming from the real cyclotomic structure of
  $\mathrm{THR}(A)$; using this we define the real topological restriction
  homology
  \[\mathrm{TRR}(A; p) \coloneqq \mathrm{holim} \left(\dotsb \xrightarrow{R} \mathrm{TRR}^{n+1}(A;
    p) \xrightarrow{R} \mathrm{TRR}^{n}(A; p) \xrightarrow{R} \dotsb
    \xrightarrow{R} \mathrm{TRR}^{1}(A; p) \right) \mathrm{.}\]
  The inclusion $C_{p^{n-1}} \le C_{p^n}$ induces a $\mathbb{Z}/2$-equivariant restriction map
  \[F : \mathrm{TRR}^{n}(A; p) = \mathrm{THR}(A)^{C_{p^n}} \to
    \mathrm{THR}(A)^{C_{p^{n-1}}} = \mathrm{TRR}^{n-1}(A; p)\mathrm{,}\]
  known as the Frobenius. This in turn induces an endomorphism $F$ on $\mathrm{TRR}(A; p)$. Finally we
  define the real topological cyclic homology of $A$ as the
  $\mathbb{Z}/2$-equivariant spectrum
  \[\mathrm{TCR}(A; p) \coloneqq \mathrm{hoeq}(\mathrm{TRR}(A; p) \,
    \substack{\xrightarrow{F}\\[-0.2em] \xrightarrow[\mathrm{id}]{}} \,
    \mathrm{TRR}(A; p)) \mathrm{.}\]
  This construction of $\mathrm{TCR}$ from $\mathrm{THR}$ is analogous to the
  usual construction of $\mathrm{TC}$ from $\mathrm{THH}$, and indeed the underlying spectrum of $\mathrm{TCR}(A; p)$ is $\mathrm{TC}(A; p)$.

As mentioned in the introduction, we will not work directly with the above
construction of $\mathrm{TCR}$, but instead use the following formula of Dotto,
Moi and Patchkoria.

\begin{theorem}[\cite{dotto_geometric_2024}~2.14] \label{thm:tcr_eq}
  We can describe topological cyclic restriction homology by the following
  homotopy equaliser
\[\mathrm{TCR}(A; 2)^{\phi \mathbb{Z}/2} \simeq \mathrm{hoeq}\left((\mathrm{THR}(A)^{\phi \mathbb{Z}/2})^{C_2} \, \substack{\xrightarrow{f}\\[-0.2em] \xrightarrow[r]{}} \, \mathrm{THR}(A)^{\phi
  \mathbb{Z}/2}\right) \mathrm{.}\]
(see \cite{dotto_geometric_2024} for descriptions of the maps, and we will later
give an explicit description in the case $A = H\mathbb{Z}/4$).
\end{theorem}

See also Appendix~A of \cite{quigley_shah_equivalence_2022}, which gives a
similar description of the geometric fixed points of $\mathrm{TCR}$ as an equalizer.

\section{Low degree homotopy of \texorpdfstring{$\mathrm{TCR}(\mathbb{Z}/4)^{\phi \mathbb{Z}/2}$}{TCR(Z/4)\^{}\{phi Z/2\}}}

In this section we compute the low degree terms of the long exact sequence in homotopy
groups corresponding to the equaliser formula of Theorem~\ref{thm:tcr_eq}, and
make some progress solving the resulting extension problems. As a result we
compute $\pi_i(\mathrm{TCR}(\mathbb{Z}/4)^{\phi \mathbb{Z}/2})$ for $i \le 1$ and
obtain short exact sequences for these homotopy groups for $2 \le i \le 5$.

\subsection{Decomposition of the equaliser formula for \texorpdfstring{$\mathrm{TCR}(\mathbb{Z}/4)^{\phi \mathbb{Z}/2}$}{TCR(Z/4)\^{}\{phi Z/2\}}} \label{sec:decomposition}

In order to compute the homotopy groups of the terms in the equaliser, we want to apply Lemma~4.3 of \cite{dotto_geometric_2024}
to decompose the copies of $\mathrm{THR}(\mathbb{Z}/4)^{\phi \mathbb{Z}/2}$. But to
satisfy the conditions of the lemma we first need to show that
$H\mathbb{Z}/4^{\phi \mathbb{Z}/2}$, endowed with the Frobenius $H\mathbb{Z}/4$-module structure,
splits as the direct sum of its homotopy groups.

\begin{lemma} \label{lem:ab_group_geo_fixed}
  For an abelian group $A$ we have
  \[\pi_i(HA^{\phi \mathbb{Z}/2}) \cong \begin{cases}
    \mathrm{coker}(A \xrightarrow{2} A) \quad &\text{$i \ge 0$, $i$ even,}\\
    \mathrm{ker}(A \xrightarrow{2} A) \quad &\text{$i \ge 0$, $i$ odd,}\\
    0 \quad &\text{$i < 0$.}\end{cases}\]
\end{lemma}
\begin{proof}
  The spectrum $HA$ is connective, so its geometric fixed points are
  too. Now consider the cofibre sequence
  \[HA_{h\mathbb{Z}/2} \to HA \to HA^{\phi
      \mathbb{Z}/2} \text{.}\]
  Since $HA$ is concentrated in degree $0$, the long exact sequence in homotopy groups gives us an exact sequence
  \begin{equation}\label{eq:ha_geo_fixed_exact} 0 \to \pi_1(HA^{\phi \mathbb{Z}/2}) \to
    \pi_0(HA_{h \mathbb{Z}/2}) \to \pi_0(HA) \to
    \pi_0(HA^{\phi \mathbb{Z}/2}) \to 0\end{equation}
  and isomorphisms
  \[\pi_i(HA^{\phi \mathbb{Z}/2}) \cong \pi_{i-1}(HA_{h \mathbb{Z}/2})\]
  for $i \ge 2$.

  But we have
  \[\pi_i(HA_{h\mathbb{Z}/2}) \cong H_i(\mathbb{RP}^{\infty}; A)
    \cong \begin{cases}A \quad &\text{$i = 0$,}\\\text{coker}(A \xrightarrow{2}
      A) \quad &\text{$i > 0$, $i$ odd,}\\\text{ker}(A \xrightarrow{2} A) \quad
      &\text{$i > 0$, $i$ even.}\end{cases}\]
  And the map $A \cong \pi_0(HA_{h \mathbb{Z}/2}) \to \pi_0(HA) \cong A$
  appearing in the exact sequence (\ref{eq:ha_geo_fixed_exact}) is
  the multiplication map $A \xrightarrow{2} A$. The result follows.
\end{proof}

\begin{lemma} \label{lem:z4_phiz2_split}
  We have a splitting of $H\mathbb{Z}/4$-modules
\[H\mathbb{Z}/4^{\phi \mathbb{Z}/2} \simeq \bigoplus_{n \ge 0} \Sigma^n H\pi_n (H\mathbb{Z}/4^{\phi
    \mathbb{Z}/2}) \cong \bigoplus_{n \ge 0} \Sigma^n H\mathbb{Z}/2 \mathrm{.}\]
  where we consider $H\mathbb{Z}/4^{\phi \mathbb{Z}/2}$ as an
  $H\mathbb{Z}/4$-module via the Frobenius
  module structure of Definition~\ref{def:frobenius}  and consider $H\mathbb{Z}/2$ as an
  $H\mathbb{Z}/4$-module via the quotient map
  $H\mathbb{Z}/4 \to H\mathbb{Z}/2$.
\end{lemma}
\begin{proof}
  First note that by Lemma~\ref{lem:ab_group_geo_fixed} we indeed have $\pi_n(H\mathbb{Z}/4^{\phi \mathbb{Z}/2}) \cong
  \mathbb{Z}/2$ for $n \ge 0$.

  Recall Lemma~\ref{lem:frob_non_derived}, which says that the Frobenius module structure on $H\mathbb{Z}/4^{\phi
    \mathbb{Z}/2}$ can be defined by the map
  \begin{equation}\label{eq:frob_structure}\varphi : H\mathbb{Z}/4 \xrightarrow{\Delta} (\tilde{N}^{\mathbb{Z}/2}_{\{e\}}
    (H\mathbb{Z}/4))^{\phi \mathbb{Z}/2} \xrightarrow{\mu^{\phi \mathbb{Z}/2}} H\mathbb{Z}/4^{\phi \mathbb{Z}/2} \mathrm{.}\end{equation}
  We claim that $\varphi$ factors along the quotient $H\mathbb{Z}/4
  \twoheadrightarrow H\mathbb{Z}/2$. This then implies the lemma, since it means
  the Frobenius $H\mathbb{Z}/4$-module structure is the pullback of a $H\mathbb{Z}/2$-module
  structure, and $\mathbb{Z}/2$ is a field so all $H\mathbb{Z}/2$-modules split uniquely as the sum of their homotopy
  groups. %

  Let the ideal $I = 2 \mathbb{Z}/4$ be the kernel of the quotient $\mathbb{Z}/4
  \twoheadrightarrow \mathbb{Z}/2$. If we were just working with discrete rings,
  then to show that a ring homomorphism out of
  $\mathbb{Z}/4$ factors through the quotient $\mathbb{Z}/2$ it would suffice to
  show that the homomorphism is zero on $I$. We will use the technology of Hovey \cite{hovey_smith_2014} to apply a
  similar argument to ring spectra.

  Considering $I$ as a non-unital ring spectrum, naturality with respect to $I
  \hookrightarrow \mathbb{Z}/4$ gives a commutative diagram
  \[\begin{tikzcd}
      HI \ar[r, "\Delta"] \ar[d] & (\tilde{N}^{\mathbb{Z}/2}_{\{e\}}
      (HI))^{\phi \mathbb{Z}/2} \ar[r, "\mu^{\phi \mathbb{Z}/2}"] \ar[d] & HI^{\phi \mathbb{Z}/2} \ar[d]\\
      H\mathbb{Z}/4 \ar[r, "\Delta"] & (\tilde{N}^{\mathbb{Z}/2}_{\{e\}}
    (H\mathbb{Z}/4))^{\phi \mathbb{Z}/2} \ar[r, "\mu^{\phi \mathbb{Z}/2}"] &
    H\mathbb{Z}/4^{\phi \mathbb{Z}/2} \mathrm{.}
  \end{tikzcd}\]
The multiplication map $\tilde{N}^{\mathbb{Z}/2}_{\{e\}} (HI) \xrightarrow{\mu}
HI$ can be expressed as
\[\tilde{N}^{\mathbb{Z}/2}_{\{e\}} (HI) \simeq (HI) \otimes_{\mathbb{S}} (HI) \to H(I \otimes I) \xrightarrow{H(\nu)} HI\]
where the first map is an equivalence of underlying spectra, the middle map is
the canonical map from lax monoidality of $H$, and $\nu : I \otimes I \to I$ is the multiplication map on $I$. But the
ring multiplication of $I$ is zero, so we conclude $\mu$ is also the zero map. Taking the geometric fixed points of the zero map gives the zero map, so in the
top row $\mu^{\phi \mathbb{Z}/2} = 0$. By commutativity of the diagram we see that the composite
$HI \hookrightarrow H\mathbb{Z}/4 \xrightarrow{\varphi} H \mathbb{Z}/4^{\phi \mathbb{Z}/2}$ is
zero (on the nose, not just null-homotopic).

  Now we apply Theorem~1.4 of \cite{hovey_smith_2014}. We have shown there is a map of Smith
  ideals
  \[\begin{tikzcd}
      HI \ar[r] \ar[d] & 0 \ar[d]\\
      H\mathbb{Z}/4 \ar[r, "\varphi"] & (H\mathbb{Z}/4)^{\phi \mathbb{Z}/2}
    \end{tikzcd}\]
  and so taking cofibres we get a commutative square of ring spectra
  \[\begin{tikzcd}
      H\mathbb{Z}/4 \ar[r, "\varphi"] \ar[d, two heads] & (H\mathbb{Z}/4)^{\phi \mathbb{Z}/2} \ar[d, equal]\\
      H\mathbb{Z}/2 \ar[r] & (H\mathbb{Z}/4)^{\phi \mathbb{Z}/2} \mathrm{,}
    \end{tikzcd}\]
  giving the desired factorisation.
\end{proof}

Now by Lemma~4.3 of \cite{dotto_geometric_2024} we get decompositions of $\mathrm{THR}(\mathbb{Z}/4)^{\phi \mathbb{Z}/2}$ and
$(\mathrm{THR}(\mathbb{Z}/4)^{\phi \mathbb{Z}/2})^{C_2}$.

\begin{corollary}\label{cor:decomp}
  We have an equivalence of $C_2$-spectra
\begin{align*}\mathrm{THR}(\mathbb{Z}/4)^{\phi \mathbb{Z}/2} \simeq &\bigoplus_{n \ge 0} \Sigma^{n \rho}
  H\mathbb{Z}/4 \otimes_{N_{\{e\}}^{C_2} (H\mathbb{Z}/4)} N_{\{e\}}^{C_2}(H\mathbb{Z}/2) \\
  & \oplus \bigoplus_{0 \le n < m} \Sigma^{n+m}(C_2)_{+} \otimes_{\mathbb{S}} H\mathbb{Z}/2
    \otimes_{H\mathbb{Z}/4} H\mathbb{Z}/2\end{align*}
  (where $\rho$ is the regular representation of $C_2$ and $(C_2)_{+}
  \otimes_{\mathbb{S}} ({-})$ denotes the left adjoint of the restriction
  functor from $C_2$-spectra to spectra) and an equivalence of spectra
  \begin{align*}
    (\mathrm{THR}(\mathbb{Z}/4)^{\phi \mathbb{Z}/2})^{C_2} \simeq &\bigoplus_{n \ge 0} \left( \Sigma^{n \rho}H\mathbb{Z}/4 \otimes_{N_{\{e\}}^{C_2}(H\mathbb{Z}/4)} N_{\{e\}}^{C_2}(H\mathbb{Z}/2) \right)^{C_2} \\
    & \oplus \bigoplus_{0 \le n < m} \Sigma^{n+m} H\mathbb{Z}/2 \otimes_{H\mathbb{Z}/4} H\mathbb{Z}/2 \mathrm{.}\end{align*}
  Note that we are using the $N_{\{e\}}^{C_2}(H\mathbb{Z}/4)$-module structure
  on $H\mathbb{Z}/4$ defined by the multiplication map $\mu :
  N_{\{e\}}^{C_2}(H\mathbb{Z}/4) \to H\mathbb{Z}/4$. \qed
\end{corollary}
\begin{remark}\label{rem:thr_underlying_decomp}
Observe that on underlying spectra the first equivalence gives
\begin{equation*} \mathrm{THR}(\mathbb{Z}/4)^{\phi \mathbb{Z}/2} \simeq \bigoplus_{n, m \ge 0}
  \Sigma^{n+m} H\mathbb{Z}/2 \otimes_{H \mathbb{Z}/4} H\mathbb{Z}/2 \mathrm{.}\end{equation*}
\end{remark}

Again by \cite{dotto_geometric_2024} Lemma~4.3, we get descriptions in
terms of these decompositions for the two maps $f, r :
(\mathrm{THR}(\mathbb{Z}/4)^{\phi \mathbb{Z}/2})^{C_2} \to \mathrm{THR}(\mathbb{Z}/4)^{\phi \mathbb{Z}/2}$
whose homotopy equaliser computes $\mathrm{TCR}(\mathbb{Z}/4)^{\phi \mathbb{Z}/2}$.

\begin{corollary}\label{cor:r_f_map_def}
  The map $f : (\mathrm{THR}(\mathbb{Z}/4)^{\phi \mathbb{Z}/2})^{C_2} \to
  \mathrm{THR}(\mathbb{Z}/4)^{\phi \mathbb{Z}/2}$ splits up according to the
  direct sum decomposition of Corollary~\ref{cor:decomp}. It is defined on the
  $n$-indexed summands by
  restriction maps
  \[\mathrm{res}_e^{C_2} : \left( \Sigma^{n \rho}H\mathbb{Z}/4
    \otimes_{N_{\{e\}}^{C_2}(H\mathbb{Z}/4)} N_{\{e\}}^{C_2}(H\mathbb{Z}/2) \right)^{C_2}
  \to \Sigma^{n \rho}H\mathbb{Z}/4 \otimes_{N_{\{e\}}^{C_2}(H\mathbb{Z}/4)}
  N_{\{e\}}^{C_2}(H\mathbb{Z}/2)\mathrm{,}\]
and on the $(n, m)$-indexed summands by diagonal maps
\[\Delta : \Sigma^{n+m} H\mathbb{Z}/2 \otimes_{H\mathbb{Z}/4} H\mathbb{Z}/2 \to \Sigma^{n+m}(C_2)_{+} \otimes_{\mathbb{S}} H\mathbb{Z}/2
  \otimes_{H\mathbb{Z}/4} H\mathbb{Z}/2 \mathrm{.}\]

  The map $r$ kills the $(n, m)$-summands of
  $(\mathrm{THR}(\mathbb{Z}/4)^{\phi \mathbb{Z}/2})^{C_2}$, and is defined on the
  $n$-summands by
  \begin{equation*}\begin{aligned}\left(\Sigma^{n \rho}H\mathbb{Z}/4 \otimes_{N_{\{e\}}^{C_2}(H\mathbb{Z}/4)}
    N_{\{e\}}^{C_2}(H\mathbb{Z}/2)\right)^{C_2} & \to \left(\Sigma^{n \rho} H\mathbb{Z}/4
    \otimes_{N_{\{e\}}^{C_2}(H\mathbb{Z}/4)} N_{\{e\}}^{C_2}(H\mathbb{Z}/2)\right)^{\phi C_2}\\
    & \simeq \Sigma^n (H\mathbb{Z}/4)^{\phi C_2} \otimes_{H \mathbb{Z}/4} H\mathbb{Z}/2\\
    & \simeq \bigoplus_{m \ge 0}
      \Sigma^{n+m} H\mathbb{Z}/2 \otimes_{H\mathbb{Z}/4} H\mathbb{Z}/2\\
      & \to \mathrm{THR}(\mathbb{Z}/4)^{\phi \mathbb{Z}/2} \end{aligned}\end{equation*}
    where the first map is the canonical map from genuine to geometric fixed
    points, the second map comes from monoidality of geometric fixed points, the
    third map is the splitting of Lemma~\ref{lem:z4_phiz2_split}, and
    the final map is the inclusion of a direct summand according to the
    decomposition of Remark~\ref{rem:thr_underlying_decomp}. \qed
  \end{corollary}

\subsection{Homotopy groups  of \texorpdfstring{$\mathrm{THR}(\mathbb{Z}/4)^{\phi \mathbb{Z}/2}$}{THR(Z/4)\^{}\{phi Z/2\}} and
  \texorpdfstring{$(\mathrm{THR}(\mathbb{Z}/4)^{\phi
      \mathbb{Z}/2})^{C_2}$}{(THR(Z/4)\^{}\{phi Z/4\})\^{}\{C\_2\}}}\label{sec:computation_of_homotopy}

  Now let us discuss the computation of the homotopy groups of
  $\mathrm{THR}(\mathbb{Z}/4)^{\phi \mathbb{Z}/2}$ and
  $(\mathrm{THR}(\mathbb{Z}/4)^{\phi \mathbb{Z}/2})^{C_2}$.

  \begin{proposition} \label{prop:pi_thr_as_sum}
    We have
\begin{equation*}\pi_i(\mathrm{THR}(\mathbb{Z}/4)^{\phi \mathbb{Z}/2}) \cong \bigoplus_{\substack{n,m,k
    \ge 0\\n+m+k=i}} \mathbb{Z}/2 \mathrm{.}\end{equation*}
\end{proposition}
\begin{proof}
  First compute $H\mathbb{Z}/2 \otimes_{H \mathbb{Z}/4} H\mathbb{Z}/2$. By the stable Dold-Kan correspondence there is a monoidal Quillen equivalence between
  $H\mathbb{Z}/4$-module spectra and chain complexes of $\mathbb{Z}/4$-modules.
  So considering a free resolution of $\mathbb{Z}/2$ as a $\mathbb{Z}/4$-module,
  we can compute the derived tensor product in chain complexes of
  $\mathbb{Z}/4$-modules and deduce
  \[H\mathbb{Z}/2 \otimes_{H \mathbb{Z}/4} H\mathbb{Z}/2 \simeq \bigoplus_{k \ge
    0} \Sigma^k H\mathbb{Z}/2 \text{.}\]
So by Corollary~\ref{cor:decomp} and Remark~\ref{rem:thr_underlying_decomp} we get an equivalence of underlying spectra
\[\mathrm{THR}(\mathbb{Z}/4)^{\phi \mathbb{Z}/2} \simeq \bigoplus_{n,m,k \ge 0} \Sigma^{n+m+k} H\mathbb{Z}/2\]
and the result follows.
\end{proof}

\begin{proposition} \label{prop:pi_thr_c2_as_sum}
  The homotopy groups of the off-diagonal part (i.e.\ the second summand) of $(\mathrm{THR}(\mathbb{Z}/4)^{\phi
  \mathbb{Z}/2})^{C_2}$ under the decomposition of Corollary~\ref{cor:decomp}
are given by
\begin{equation*}\pi_i\left( \bigoplus_{0 \le n < m} \Sigma^{n+m} H\mathbb{Z}/2 \otimes_{H\mathbb{Z}/4}
  H\mathbb{Z}/2 \right) \cong
  \bigoplus_{\substack{n,m,k \ge 0\\n < m\\n+m+k=i}} \mathbb{Z}/2 \mathrm{.}\end{equation*}
\end{proposition}
\begin{proof}
  Similarly to Proposition~\ref{prop:pi_thr_as_sum}, we have an equivalence of spectra
\[ \bigoplus_{0 \le n < m} \Sigma^{n+m} H\mathbb{Z}/2 \otimes_{H\mathbb{Z}/4}
  H\mathbb{Z}/2 \simeq \bigoplus_{0 \le n < m} \bigoplus_{k \ge 0}
  \Sigma^{n+m+k} H\mathbb{Z}/2\]
giving the result.
\end{proof}

So we just need to focus on calculating the homotopy groups of the diagonal
summands of $(\mathrm{THR}(\mathbb{Z}/4)^{\phi \mathbb{Z}/2})^{C_2}$, which are of the form
\[\left(\Sigma^{n \rho}H\mathbb{Z}/4 \otimes_{N_{\{e\}}^{C_2}(H\mathbb{Z}/4)}
  N_{\{e\}}^{C_2}(H\mathbb{Z}/2)\right)^{C_2}\]
for $n \ge 0$. We approach this using the concept of non-abelian derived functors,
as defined in \cite{dold_non-additive_1958}. First let us recall the definition.

  Let $R$ be a commutative ring and let $T : \mathrm{Mod}_R \to \mathrm{Ab}$ be a functor with
  $T(0) = 0$. Let $i \ge 0$ and $n \ge 0$ be natural numbers. Given an $R$-module $M$, let $M_\bullet$ be a level-wise free
  simplicial $R$-module such that the geometric realisation of the $HR$-module spectrum $HM_{\bullet}$ is
  $\Sigma^n HM$ (in the following we will sometimes abuse terminology and refer to
  $\abs{HM_{\bullet}}$ as the geometric realisation of $M_{\bullet}$, implicitly
  taking the Eilenberg-MacLane spectrum levelwise). Then we define a functor
  \[L_i^{(n)}T : \mathrm{Mod}_R \to \mathrm{Ab}\]
  via
  \[L_i^{(n)}T(M) \coloneqq \pi_i(\abs{HT(M_\bullet)})\]
  where $T(M_\bullet)$ means the simplicial abelian group produced by applying
  $T$ level-wise to $M_\bullet$. Note one can always produce such an
  $M_\bullet$, for example by taking a free resolution for $M[n]$ in chain complexes of $R$-modules, then
  applying the Dold-Kan equivalence of Theorem~\ref{thm:dold_kan}.

\begin{definition}[Non-abelian derived functor]\label{def:non-ab_derived}
  We call $L_i^{(n)}T$ the $i$th non-abelian derived functor of type $n$ of $T$.
\end{definition}
\begin{remark}\label{rem:nonab_i_less_n}
  Observe that $T(M_\bullet)$ is zero in degrees less than $n$, so $L_i^{(n)}T(M) = 0$ for $i < n$.
\end{remark}

\begin{lemma} \label{lem:expansion_to_em}
  Let $R$ be a commutative ring and $M$ an $R$-module. Let $M_{\bullet}$ be a free
  simplicial $R$-module with geometric realisation $\Sigma^n HM$. Then
  \begin{align*}
    \Sigma^{n\rho} HR \otimes_{N_{\{e\}}^{C_2}(HR)} N_{\{e\}}^{C_2}(HM)
    &\simeq \abs{H\left( M_\bullet \otimes_R M_\bullet, w \right)}\mathrm{,}
  \end{align*}
  where $w$ is the involution on $M_i \otimes_R M_i$
  given by swapping the two copies of $M_i$. This equivalence is natural in
  $M_\bullet$ (if we take $M$ to depend functorially on $M_\bullet$).
\end{lemma}
\begin{proof}
  First rewrite as
  \begin{align*}
    \Sigma^{n\rho} HR \otimes_{N_{\{e\}}^{C_2}(HR)} N_{\{e\}}^{C_2}(HM) &\simeq HR \otimes_{N_{\{e\}}^{C_2}(HR)} N_{\{e\}}^{C_2}(\Sigma^nHM)\\
    &\simeq \abs{HR \otimes_{N_{\{e\}}^{C_2}(HR)} N_{\{e\}}^{C_2}(HM_\bullet)} \mathrm{.}
  \end{align*}
  To see the first equivalence, write the suspension $\Sigma^n({-})$ as a smash product with the
  sphere spectrum $S^n \wedge ({-})$, then note the norm is monoidal by
  \cite{hill_nonexistence_2016-1} Proposition A.53 and $N^{C_2}_{\{e\}}(S^n) \cong S^{n \rho}$ by
  \cite{hill_nonexistence_2016-1} Proposition A.59. The second equivalence uses
  the fact that geometric realisation commutes with the norm and tensor products
  (recall by \cite{hill_nonexistence_2016-1} Proposition A.53 the norm commutes with
  sifted colimits).

  Let $B_i$ index a basis of the free module $M_i$, i.e.\ $M_i \cong \bigoplus_{B_i} R$.
  We first establish some notation. Consider the set $(B_i \times B_i)_{C_2}$, where $C_2$ acts on $B_i
  \times B_i$ by swapping the two factors. Denote the diagonal part by $B^d_i = \{(u, u) \in (B_i \times B_i)_{C_2}\} \cong
  B_i$ and denote the off-diagonal part by $B^o_i = (B_i \times B_i)_{C_2} \setminus B^d_i$.
  Then we have a series of equivalences of simplicial $C_2$-spectra
  \begin{align*}
    &\phantom{{}\simeq{}} HR \otimes_{N_{\{e\}}^{C_2}(HR)} N_{\{e\}}^{C_2}(HM_\bullet)\\
    &\simeq HR \otimes_{N_{\{e\}}^{C_2}(HR)} N_{\{e\}}^{C_2}\left(\bigoplus_{B_{\bullet}} HR\right)\\
    &\simeq HR \otimes_{N_{\{e\}}^{C_2}(HR)} \left( \bigoplus_{B^d_{\bullet}} N_{\{e\}}^{C_2}(HR) \oplus \bigoplus_{B^o_\bullet} \left(\Sigma^\infty(C_2)_{+} \otimes_{\mathbb{S}} N_{\{e\}}^{C_2}(HR) \right)\right)\\
    &\simeq \bigoplus_{B^d_{\bullet}} HR \oplus \bigoplus_{B^o_\bullet} \left( \Sigma^\infty(C_2)_{+} \otimes_{\mathbb{S}} HR \right)\\
    &\simeq H\left(\bigoplus_{B_\bullet \times B_\bullet} R, w' \right) \\
    &\simeq H\left( M_\bullet \otimes_R M_\bullet, w \right)\mathrm{,}
  \end{align*}
  where $w'$ is the involution on $\bigoplus_{B_i \times B_i} R$
  given by swapping the two copies of $B_i$, and which under $\bigoplus_{B_i \times
    B_i} R \cong M_i \otimes_R M_i$ becomes the involution $w$ described in the
  lemma statement.

  To see naturality, consider the map of spectra
\[HR \otimes_{N_{\{e\}}^{C_2}(HR)} N_{\{e\}}^{C_2}(HM_i) \to
  H(M_i \otimes M_i)\]
obtained as the canonical map from the left hand side to the Eilenberg-MacLane spectrum
of its degree zero homotopy, and observe that taking geometric realisations gives the
map in the statement of the lemma; from this definition of the map it is clear
that it is natural in $M_\bullet$.
\end{proof}

The immediate consequence of this theorem is that we can describe some homotopy
groups of interest using non-abelian derived functors. We present these
generally but note the case we will be interested in is when
  $R = \mathbb{Z}/4$ and $M = \mathbb{Z}/2$.
  First we consider the
homotopy groups of the genuine $C_2$-fixed points; this appeared as
Lemma~\ref{lem:intro_nonab} in the introduction.
Define $F_R : \mathrm{Mod}_R \to \mathrm{Ab}$ on an
  $R$-module $M$ by $F_R(M)
  = (M \otimes_R M)^{C_2}$, where $M \otimes_R M$ has the $C_2$-action
  given by swapping the factors.

\begin{corollary}\label{cor:F_non_ab}
  We have
  \[\pi_i^{C_2}\Big(\Sigma^{n\rho} HR
      \otimes_{N_{\{e\}}^{C_2}(HR)}
      N^{C_2}_{\{e\}}(HM)\Big) \cong L^{(n)}_i F_R(M)\]
  where $L^{(n)}_i F_R$ denotes the $i$th non-abelian derived functor of type $n$
  of $F_R$.
\end{corollary}
\begin{proof}
  We have
  \begin{align*}
    \pi_i^{C_2}\Big(\Sigma^{n\rho} HR \otimes_{N_{\{e\}}^{C_2}(HR)} N_{\{e\}}^{C_2}(HM)\Big)
    &\cong \pi_i^{C_2}\big(\abs{H(M_\bullet \otimes_{R} M_\bullet, w)}\big)\\
    &\cong \pi_i\big(\abs{HF_R(M_{\bullet})}\big)\\
    &\cong L^{(n)}_iF_R(M) \mathrm{.} \qedhere
  \end{align*}
\end{proof}

Next we consider the homotopy groups of the underlying non-equivariant spectrum.
Define $G_R : \mathrm{Mod}_R \to \mathrm{Ab}$ on an
  $R$-module $M$ by $G_R(M) = M \otimes_R M$.
\begin{corollary}\label{cor:G_non_ab}
  We have
  \[\pi_i\Big(\Sigma^{n\rho} HR \otimes_{N_{\{e\}}^{C_2}(HR)} N^{C_2}_{\{e\}}(HM)\Big) \cong L^{(n)}_i G_R(M)\]
  where $L^{(n)}_i G_R$ denotes the $i$th non-abelian derived functor of type $n$
  of $G_R$.
\end{corollary}
\begin{proof}
  We have
  \begin{align*}
    \pi_i\Big(\Sigma^{n\rho} HR \otimes_{N_{\{e\}}^{C_2}(HR)} N_{\{e\}}^{C_2}(HM)\Big)
    &\cong \pi_i\big(\abs{H(M_\bullet \otimes_{R} M_\bullet)}\big)\\
    &\cong L^{(n)}_iG_R(M) \mathrm{.} \qedhere
  \end{align*}
\end{proof}

However this time when $R = \mathbb{Z}/4$ and $M = \mathbb{Z}/2$ we can also
compute the homotopy groups directly, giving us a computation of
$L^{(n)}_iG_{\mathbb{Z}/4}(\mathbb{Z}/2)$.

\begin{lemma} \label{lem:l_g_comp}
  We have
  \[L^{(n)}_iG_{\mathbb{Z}/4}(\mathbb{Z}/2) \cong \begin{cases}0 & \text{if $i < 2n$}\\ \mathbb{Z}/2 & \text{if $i \ge
                                                             2n$.}\end{cases} \]
\end{lemma}
\begin{proof}
  On underlying spectra we have
  \begin{align*}
    \Sigma^{n\rho} H\mathbb{Z}/4 \otimes_{N_{\{e\}}^{C_2}(H\mathbb{Z}/4)} N^{C_2}_{\{e\}}(H\mathbb{Z}/2) &\cong \Sigma^{2n} (H\mathbb{Z}/2 \otimes_{H\mathbb{Z}/4} H\mathbb{Z}/2)\\
    &\cong \bigoplus_{k \ge 0} \Sigma^{2n+k} H\mathbb{Z}/2
  \end{align*}
  and observe
  \[\pi_i\left(\bigoplus_{k \ge 0} \Sigma^{2n+k} H\mathbb{Z}/2\right)
  \cong \begin{cases}0 & \text{if $i < 2n$}\\ \mathbb{Z}/2 & \text{if $i \ge
                                                             2n$.}\end{cases} \]
  Now by Corollary~\ref{cor:G_non_ab} the result follows.
\end{proof}

For small $i$ and $n$ we can compute $L_i^{(n)}F_{\mathbb{Z}/4}(\mathbb{Z}/2)$
directly with computer assistance.
We work with the free resolution of $\mathbb{Z}/2[n]$ in chain complexes given
by
\[\dotsb \xrightarrow{2} \mathbb{Z}/4 \xrightarrow{2} \mathbb{Z}/4
  \xrightarrow{2} \mathbb{Z}/4 \to 0 \to \dotsb \mathrm{,}\]
where the rightmost $\mathbb{Z}/4$ is in degree $n$. Applying the Dold-Kan
equivalence (recalled in Theorem~\ref{thm:dold_kan}) between chain complexes and simplicial modules
gives
a level-wise free simplicial $\mathbb{Z}/4$-module $M_\bullet$ with geometric
realisation $\Sigma^n H\mathbb{Z}/2$. We use this to compute the
non-abelian derived functor exactly as defined in Definition~\ref{def:non-ab_derived}: apply $F_{\mathbb{Z}/4}$ level-wise, then compute homology groups of the Moore complex.

Without any
particularly sophisticated optimisation and using a normal laptop computer, we managed
to compute just far enough to prove the non-periodicity result mentioned in
the introduction. However the simplicial structure coming
from Dold-Kan gets exponentially larger as the degree increases, so even with more serious computing power we wouldn't be
able to go much farther without significant algorithmic improvements.

Combining the computations from the computer program with our earlier discussion, we obtain the following.
These results appeared as Theorem~\ref{thm:intro_thr_groups} in the introduction.

\begin{theorem} \label{thm:thr_groups}
  The low degree homotopy groups of the terms in the equaliser of Theorem~\ref{thm:tcr_eq}
   computing $\mathrm{TCR}(\mathbb{Z}/4)^{\phi \mathbb{Z}/2}$
  are as follows:
\[\begin{tabular}{>{$}c<{$}>{$}c<{$}>{$}c<{$}}
  i & \pi_i((\mathrm{THR}(\mathbb{Z}/4)^{\phi \mathbb{Z}/2})^{C_2}) & \pi_i(\mathrm{THR}(\mathbb{Z}/4)^{\phi \mathbb{Z}/2})\\ \hline
  0 & \mathbb{Z}/4 & \mathbb{Z}/2\\
  1 & (\mathbb{Z}/2)^4 & (\mathbb{Z}/2)^3\\
  2 & (\mathbb{Z}/2)^7 & (\mathbb{Z}/2)^6\\
  3 & (\mathbb{Z}/2)^{13} & (\mathbb{Z}/2)^{10}\\
  4 & (\mathbb{Z}/2)^{15} \oplus (\mathbb{Z}/4)^3 & (\mathbb{Z}/2)^{15}\\
  5 & (\mathbb{Z}/2)^{27} & (\mathbb{Z}/2)^{21}\\
  6 & (\mathbb{Z}/2)^{34} & (\mathbb{Z}/2)^{28}
\end{tabular}\]
\end{theorem}
\begin{proof}
  Proposition~\ref{prop:pi_thr_as_sum} computed
  $\pi_i(\mathrm{THR}(\mathbb{Z}/4)^{\phi \mathbb{Z}/2})$. For
  $(\mathrm{THR}(\mathbb{Z}/4)^{\phi \mathbb{Z}/2})^{C_2}$ we work with the
  decomposition of Corollary~\ref{cor:decomp}.
  Proposition~\ref{prop:pi_thr_c2_as_sum} computed the homotopy groups of the off-diagonal part of the
  decomposition. We use a computer program to compute the homotopy groups of the $n$-indexed
  diagonal summands of the decomposition,
  which by Corollary~\ref{cor:F_non_ab} are given by the values of the non-abelian
  derived functors $L_i^{(n)}F_R$ at $\mathbb{Z}/2$.

  For example at $i = 1$ the computer program calculates
  \[L_1^{(0)}F_{\mathbb{Z}/4}(\mathbb{Z}/2) \cong (\mathbb{Z}/2)^2 \quad \quad
    L_1^{(1)}F_{\mathbb{Z}/4}(\mathbb{Z}/2) \cong \mathbb{Z}/2\]
  (and $L_1^{(n)}F_{\mathbb{Z}/4}(\mathbb{Z}/2) = 0$ for $n > 1$) so we get
  \begin{align*}\pi_i((\mathrm{THR}(\mathbb{Z}/4)^{\phi \mathbb{Z}/2})^{C_2}) &\cong
    \bigoplus_{0 \le n \le 1} L_1^{(n)}F_{\mathbb{Z}/4}(\mathbb{Z}/2) \oplus \bigoplus_{\substack{n,m,k
                                                                                \ge 0\\n < m\\n+m+k=1}} \mathbb{Z}/2\\
                                                                              &\cong (\mathbb{Z}/2)^2 \oplus
                                                                                \mathbb{Z}/2 \oplus \mathbb{Z}/2\\
    &\cong (\mathbb{Z}/2)^4 \text{.}\end{align*}
  See the code repository for the full output for all $i$.
\end{proof}

\subsection{Action of the maps \texorpdfstring{$f$}{f} and \texorpdfstring{$r$}{r} on homotopy}\label{sec:action_on_htpy}

We can go further and compute the maps $f$ and $r$ on the homotopy groups that
we have calculated,
allowing us to obtain short exact sequences describing the homotopy groups of $\mathrm{TCR}(\mathbb{Z}/4)^{\phi
  \mathbb{Z}/2}$. Recall the descriptions of $f$ and $r$ in
Corollary~\ref{cor:r_f_map_def}, defined in terms of the decompositions in
Corollary~\ref{cor:decomp}. The actions of the maps on the $(n, m)$-summands of
$(\mathrm{THR}(\mathbb{Z}/4)^{\phi \mathbb{Z}/2})^{C_2}$
are
straightforward: $f$ acts by the diagonal map, and $r$ is zero. This section
is mostly occupied by analysing what happens on the diagonal $n$-indexed summands.

First let us consider how $f$ acts on the homotopy groups of the diagonal summands.
For each $n \ge 0$, the map $f$ takes the $n$-indexed summand of $(\mathrm{THR}(\mathbb{Z}/4)^{\phi
\mathbb{Z}/2})^{C_2}$ to the $n$-indexed summand of $\mathrm{THR}(\mathbb{Z}/4)^{\phi
\mathbb{Z}/2}$ by the restriction map.
Recall from Corollary~\ref{cor:F_non_ab}  that the degree $i$ homotopy of the $n$-indexed summand of $(\mathrm{THR}(\mathbb{Z}/4)^{\phi
\mathbb{Z}/2})^{C_2}$ is isomorphic to
$L_i^{(n)}F_{\mathbb{Z}/4}(\mathbb{Z}/2)$, and recall from Corollary \ref{cor:G_non_ab} and
Lemma~\ref{lem:l_g_comp} that the degree $i$
homotopy of the $n$-indexed summand of $\mathrm{THR}(\mathbb{Z}/4)^{\phi
  \mathbb{Z}/2}$ is isomorphic to $L_i^{(n)}G_{\mathbb{Z}/4}(\mathbb{Z}/2)$, which is $0$ for
$i < 2n$ and $\mathbb{Z}/2$ for $i \ge 2n$. Then we find:

\begin{proposition} \label{prop:f_on_diag_comp}
  By Corollary~\ref{cor:r_f_map_def} the map $f$ restricted to the $n$-indexed
  diagonal summands of the source and target is the restriction
  map
  \[\pi^{C_2}_i\Big(\Sigma^{n\rho}H\mathbb{Z}/4
      \otimes_{N^{C_2}_{\{e\}}(H\mathbb{Z}/4)}N^{C_2}_{\{e\}}(H\mathbb{Z}/2)\Big)
    \to \pi_i\Big(\Sigma^{n\rho}H\mathbb{Z}/4
      \otimes_{N^{C_2}_{\{e\}}(H\mathbb{Z}/4)}N^{C_2}_{\{e\}}(H\mathbb{Z}/2)\Big) \text{.}\]
  Under the isomorphisms between these groups and non-abelian derived functors this becomes the map
  \[L_i^{(n)}F_{\mathbb{Z}/4}(\mathbb{Z}/2) \to L_i^{(n)}G_{\mathbb{Z}/4}(\mathbb{Z}/2)\]
  induced by the natural transformation $F_{R} \Rightarrow
  G_{R}$ given at a module $M$ by inclusion of fixed points $(M \otimes_{R} M)^{C_2}
  \hookrightarrow M \otimes_R M$.
\end{proposition}
\begin{proof}
  Let $M_\bullet$ be a free simplicial $\mathbb{Z}/4$-module with realisation
  $\Sigma^n\mathbb{Z}/2$. Under the equivalence of Lemma~\ref{lem:expansion_to_em} the restriction map on the $n$-indexed
  summands becomes the restriction map
  \[\pi_i^{C_2}\big(\abs{H(M_\bullet \otimes_{\mathbb{Z}/4} M_\bullet,
      w)}\big) \to \pi_i\big(\abs{H(M_\bullet \otimes_{\mathbb{Z}/4} M_\bullet,
      w)}\big)\]
  which is equivalently the map
  \[\pi_i\big(\abs{H(F_{\mathbb{Z}/4}(M_\bullet))}\big) \to \pi_i\big(\abs{H(G_{\mathbb{Z}/4}(M_\bullet))}\big)\]
    induced by the inclusion of fixed points natural transformation $F_R
    \Rightarrow G_R$. But this is precisely the definition of the map on
    non-abelian derived functors induced by $F_R \Rightarrow G_R$.
\end{proof}

We can straightforwardly extend our computer program to compute this map on non-abelian
derived functors. Moreover there is no problem identifying the source and target
with the appropriate summands of the groups computed in Theorem~\ref{thm:thr_groups}: we
computed $\pi_i((\mathrm{THR}(\mathbb{Z}/4)^{\phi \mathbb{Z}/2})^{C_2})$ using
non-abelian derived functors in the first place, and $\mathbb{Z}/2$ has no
non-trivial automorphisms so for $i \ge 2n$ we can immediately identify
$L_i^{(n)}G_{\mathbb{Z}/4}(\mathbb{Z}/2)$ with the appropriate summand of
$\pi_i(\mathrm{THR}(\mathbb{Z}/4)^{\phi \mathbb{Z}/2})$ (namely the $(n, m, k) =
(n, n,
i-2n)$-indexed summand under the indexing of Proposition~\ref{prop:pi_thr_as_sum}).

It remains to compute the action of the map $r$, which proves to be more involved. Recall our
description of $r$ in Corollary~\ref{cor:r_f_map_def}. Let $M_\ast$ be
the standard free resolution of $\mathbb{Z}/2[n]$ by a chain complex of free $\mathbb{Z}/4$-modules
\begin{equation*}\label{eq:standard_free_res}\dotsb \to \mathbb{Z}/4 \xrightarrow{\cdot 2} \mathbb{Z}/4 \xrightarrow{\cdot 2}
  \mathbb{Z}/4 \to 0 \to \dotsb \mathrm{,}\end{equation*}
and let
  $\sigma : \mathrm{Ch}_{\ge 0}(\mathbb{Z}/4) \to \mathrm{sMod}_{\mathbb{Z}/4}$
be the Dold-Kan equivalence functor recalled in Theorem~\ref{thm:dold_kan},
so we have $\sigma M_{\bullet} \coloneqq \sigma(M_\ast)$ a simplicial $\mathbb{Z}/4$-module with geometric realisation $\Sigma^n H\mathbb{Z}/2$.
Rewriting the description of $r$ in terms of this free simplicial resolution of
$M$, the action of $r$ on the $n$-indexed diagonal component of
$(\mathrm{THR}(\mathbb{Z}/4)^{\phi \mathbb{Z}/2})^{C_2}$ is given by
\begin{equation} \label{eq:r_map_simp_def}
\begin{aligned} \abs{\left(H\mathbb{Z}/4 \otimes_{N_{\{e\}}^{C_2}(H\mathbb{Z}/4)}
    N_{\{e\}}^{C_2}(H\sigma M_{\bullet})\right)^{C_2}} & \to \abs{\left(H\mathbb{Z}/4
    \otimes_{N_{\{e\}}^{C_2}(H\mathbb{Z}/4)} N_{\{e\}}^{C_2}(H\sigma M_{\bullet})\right)^{\phi C_2}}\\
    & \simeq \abs{H\mathbb{Z}/4^{\phi C_2}
      \otimes_{H \mathbb{Z}/4} H\sigma M_{\bullet} }\\
    & \simeq H\mathbb{Z}/4^{\phi C_2}
      \otimes_{H \mathbb{Z}/4} \abs{H\sigma M_{\bullet} }\\
    & \simeq H\mathbb{Z}/4^{\phi C_2}
      \otimes_{H \mathbb{Z}/4} \Sigma^n H\mathbb{Z}/2\\
    & \simeq \bigoplus_{m \ge 0}
      \Sigma^{n+m} H\mathbb{Z}/2 \otimes_{H\mathbb{Z}/4} H\mathbb{Z}/2\\
                                                       & \to \mathrm{THR}(\mathbb{Z}/4)^{\phi \mathbb{Z}/2} \mathrm{,}
\end{aligned}\end{equation}
where the first map is the geometric realisation of the canonical map from genuine to geometric
fixed points, the second map comes from monoidality of geometric fixed
points, and the remaining maps are all the evident canonical maps. The source is
the geometric realisation of a (non-discrete) simplicial spectrum, not just a simplicial abelian group as
before, so we can no longer straightforwardly plug it into the computer; we
will need to use some tricks.

We start by simplifying the map from genuine to geometric fixed points.

\begin{lemma}\label{lem:genuine_to_geometric_expand}
Write $\sigma M_{\bullet} = \bigoplus_{B_{\bullet}} \mathbb{Z}/4$.
  Then we have a commutative diagram of simplicial spectra
  \[\begin{tikzcd}
\left(H\mathbb{Z}/4 \otimes_{N_{\{e\}}^{C_2}(H\mathbb{Z}/4)}
    N_{\{e\}}^{C_2}(H\sigma M_{\bullet})\right)^{C_2} \ar[r] \ar[d, "{\cong}"] & \left(H\mathbb{Z}/4
    \otimes_{N_{\{e\}}^{C_2}(H\mathbb{Z}/4)} N_{\{e\}}^{C_2}(H\sigma
    M_{\bullet})\right)^{\phi C_2} \ar[d, "{\cong}"]\\
  \bigoplus_{(B_\bullet \times B_\bullet)_{C_2}} H\mathbb{Z}/4 \ar[r] &
  \bigoplus_{B_\bullet} H\mathbb{Z}/4^{\phi C_2}
    \end{tikzcd}\]
  where the bottom horizontal map is defined by the projection
  \[\bigoplus_{(B_\bullet \times B_\bullet)_{C_2}} H\mathbb{Z}/4 \to
    \bigoplus_{B^d_\bullet} H\mathbb{Z}/4 \cong \bigoplus_{B_\bullet} H\mathbb{Z}/4\]
  onto the diagonal components of $(B_\bullet \times B_\bullet)_{C_2}$ (recall that we
split the set $(B_i \times B_i)_{C_2}$ into diagonal and off-diagonal parts as
$B_i^d \amalg B^o_i$) composed with
  the map
  \[\bigoplus_{B_\bullet} H\mathbb{Z}/4 \to \bigoplus_{B_\bullet}
    H\mathbb{Z}/4^{\phi C_2}\]
  induced component-wise by the canonical map $H\mathbb{Z}/4 \cong H\mathbb{Z}/4^{C_2} \to
  H\mathbb{Z}/4^{\phi C_2}$. The simplicial structure on the bottom right group
  is such that the isomorphism
  \[\bigoplus_{B_\bullet} H\mathbb{Z}/4^{\phi C_2} \cong H\mathbb{Z}/4^{\phi
      C_2} \otimes_{H\mathbb{Z}/4} \bigoplus_{B_\bullet} H\mathbb{Z}/4 \cong
    H\mathbb{Z}/4^{\phi C_2} \otimes_{H\mathbb{Z}/4} H\sigma M_\bullet\]
  is an isomorphism of simplicial spectra, where we consider
  $H\mathbb{Z}/4^{\phi C_2}$ as an $H\mathbb{Z}/4$-module via the Frobenius
  module structure.
\end{lemma}
\begin{proof}
  Recalling the proof of Lemma~\ref{lem:expansion_to_em} we have
  an equivalence of simplicial $C_2$-spectra %
  \begin{align*}\left(H\mathbb{Z}/4 \otimes_{N_{\{e\}}^{C_2}(H\mathbb{Z}/4)}
    N_{\{e\}}^{C_2}(H\sigma M_{\bullet})\right)^{C_2} &\cong
    \left(\bigoplus_{B^d_\bullet} H\mathbb{Z}/4 \oplus  \bigoplus_{B^o_\bullet} (C_2)_{+} \otimes_{\mathbb{S}} H\mathbb{Z}/4 \right)^{C_2}\\
    &\cong H\left( \bigoplus_{B_{\bullet} \times B_{\bullet}} \mathbb{Z}/4, w' \right)^{C_2} \\
    &\cong HF_{\mathbb{Z}/4}(\sigma M_\bullet) \mathrm{,}
  \end{align*}
  where $w'$ is the involution on $\bigoplus_{B_\bullet \times B_\bullet}
  \mathbb{Z}/4$ corresponding to interchanging the copies of $B_\bullet$, and $F_{\mathbb{Z}/4}(N) \coloneqq (N \otimes_{\mathbb{Z}/4} N)^{C_2}$.
  We have
  \[F_{\mathbb{Z}/4}((\sigma M)_i) \cong \bigoplus_{(B_{i} \times B_{i})_{C_2}} \mathbb{Z}/4
    \cong \bigoplus_{B^d_i} \mathbb{Z}/4 \oplus \bigoplus_{B^o_i} \mathbb{Z}/4
    \mathrm{,}\]
  and observe that $\bigoplus_{B^d_i} \mathbb{Z}/4$ is a sub-simplicial
  abelian group. In fact if we let $M^0_{\ast}$ be
  the chain complex
  \[\dotsb \mathbb{Z}/4 \xrightarrow{\cdot 0} \mathbb{Z}/4 \xrightarrow{\cdot 0}
    \mathbb{Z}/4 \to 0\]
  (with the lowest degree $\mathbb{Z}/4$ in degree $n$) then
  $\bigoplus_{B^d_i} \mathbb{Z}/4 \cong \sigma M^0_{\bullet}$. However the
  level-wise subgroups $\bigoplus_{B^o_i} \mathbb{Z}/4$ do
  not define a sub-simplicial abelian group, since face maps applied to off-diagonal
  terms can hit diagonal terms. That is, our level-wise decomposition of $\bigoplus_{(B_i \times
    B_i)_{C_2}} \mathbb{Z}/4$ does not give a decomposition of simplicial abelian groups.

  Similarly analysing the geometric fixed points, we get an equivalence of
  simplicial spectra %
  \begin{equation} \label{eq:geo_fixed}
    \begin{aligned}
      \left(H\mathbb{Z}/4 \otimes_{N_{\{e\}}^{C_2}(H\mathbb{Z}/4)}
      N_{\{e\}}^{C_2}(H\sigma M_{\bullet})\right)^{\phi C_2} &\cong
    \left(\bigoplus_{B^d_\bullet} H\mathbb{Z}/4 \oplus  \bigoplus_{B^o_\bullet} \Sigma^\infty(C_2)_{+} \otimes_{\mathbb{S}} H\mathbb{Z}/4 \right)^{\phi C_2}\\
      &\cong \bigoplus_{B^d_\bullet} H\mathbb{Z}/4^{\phi C_2}\\
                                                             &\cong  H\mathbb{Z}/4^{\phi C_2} \otimes_{H \mathbb{Z}/4} H\sigma M_\bullet\mathrm{,}
    \end{aligned}
  \end{equation}
  where in the last line we take $H \mathbb{Z}/4^{\phi C_2}$ with the Frobenius
  $H\mathbb{Z}/4$-module
  structure.
  Note that this time all the off-diagonal components vanish.
  We should briefly explain how the Frobenius module structure arises here. The
  copies of $H\mathbb{Z}/4^{\phi C_2}$ arise from the isomorphism
  \[H\mathbb{Z}/4^{\phi C_2} \cong \left(H \mathbb{Z}/4
      \otimes_{N_{\{e\}}^{C_2}(H\mathbb{Z}/4)} N_{\{e\}}^{C_2}(H \mathbb{Z}/4)\right)^{\phi
        C_2} \mathrm{.}\]
    The simplicial structure maps involve acting (with the usual
    $H\mathbb{Z}/4$-module structure) on the copy of
    $H\mathbb{Z}/4$ contained in the norm on the right hand side. But under the
    isomorphism this is exactly the action in the Frobenius module structure on
    $H\mathbb{Z}/4^{\phi C_2}$.

  Finally observe that under the above equivalences of spectra, the map from genuine to geometric fixed points of $H\mathbb{Z}/4 \otimes_{N_{\{e\}}^{C_2}(H\mathbb{Z}/4)}
  N_{\{e\}}^{C_2}(H\sigma M_{\bullet})$ becomes the map
  \begin{equation*}\bigoplus_{(B_{\bullet} \times B_{\bullet})_{C_2}} H\mathbb{Z}/4 \to \bigoplus_{B_{\bullet}}
    H\mathbb{Z}/4^{\phi C_2}
  \end{equation*}
  as described in the statement of the Lemma.
\end{proof}

    Now we need to understand how the map $r$ acts on homotopy classes
    as computed by our computer
    program. Recall that our program works by computing the homology groups of the
    Moore complex of the simplicial abelian group
    \[F_{\mathbb{Z}/4}(\sigma M_\bullet) \cong \bigoplus_{(B_{\bullet} \times B_{\bullet})_{C_2}}
        \mathbb{Z}/4 \text{.}\]
    Fix some $i \ge n$, and let $\partial_i$ denote the degree $i$ differential
    of the Moore complex. A homology class $[x] \in
    \mathrm{ker}(\partial_i)/\mathrm{im}(\partial_{i+1}) \cong
    \pi_i\left(\abs{\bigoplus_{(B_\bullet \times B_\bullet)_{C_2}}
        H\mathbb{Z}/4}\right)$ is represented by some $x \in
    \mathrm{ker}(\partial_i) \le \bigoplus_{(B_i \times B_i)_{C_2}} \mathbb{Z}/4$.
    We want to determine the image
    $r_\ast([x])$.

    \begin{lemma}
      \begin{enumerate}[(i)]
        \item
      If we take a purely off-diagonal element
  $x \in \mathrm{ker}(\partial_i) \cap \bigoplus_{B^o_i} \mathbb{Z}/4$, then $r_\ast([x]) = 0$.
  \item
    For $x \in \text{ker}(\partial_i)$, $r_\ast([2x]) = 0$.
    \end{enumerate}
  \end{lemma}
  \begin{proof}
  Part (i) is clear from the description of the map in
  Lemma~\ref{lem:genuine_to_geometric_expand}. To be fully rigorous, consider the simplicial map from the
  $i$-simplex $\Delta^i \to \bigoplus_{(B_\bullet \times B_\bullet)_{C_2}}
  H\mathbb{Z}/4$ representing $[x]$, and observe that the postcomposition with
  the projection onto the diagonal $\bigoplus_{B^d_\bullet} H\mathbb{Z}/4$ is the
  zero map.

  For (ii), observe that $\pi_i(\mathrm{THR}(\mathbb{Z}/4)^{\phi \mathbb{Z}/2})$ is
  entirely $2$-torsion, so for any $x \in \mathrm{ker}(\partial_i)$ we have
  $r_\ast([2x]) = 2r_\ast([x]) = 0$.
  \end{proof}
\begin{remark}\label{rem:suffices_residue_class}
  So for a general $x \in \mathrm{ker}(\partial_i)$, in order to compute
  $r_\ast([x])$ it suffices to analyse the residue class of $x$ in
  \[\mathrm{ker}(\partial_i)/(\mathrm{im}(\partial_{i+1}) +
    \mathrm{ker}(\partial_i)\cap \bigoplus_{B^o_i} \mathbb{Z}/4 + 2\mathrm{ker}(\partial_i)) \mathrm{.}\]
\end{remark}

  Let $e_{(u, v)} \in \bigoplus_{(B_i \times B_i)_{C_2}}
  \mathbb{Z}/4$ denote the basis element corresponding to $(u, v) \in (B_i \times
  B_i)_{C_2}$. Recalling the Dold-Kan equivalence of Theorem~\ref{thm:dold_kan}, explicitly $B_i$ is the set of surjective
  maps $[i] \twoheadrightarrow [j]$ where $n \le j \le i$, so there is an element
  $\mathrm{id}_{[i]} \in B_i$ corresponding to the identity map.

  \begin{lemma} \label{lem:quotient_isomorphism}
    We have an isomorphism
    \[\mathrm{ker}(\partial_i)/(\mathrm{im}(\partial_{i+1}) +
      \mathrm{ker}(\partial_i)\cap \bigoplus_{B^o_i} \mathbb{Z}/4 + 2\mathrm{ker}(\partial_i)) \cong \mathbb{Z}/2 \mathrm{,}\]
    where the left hand side is generated by the residue class of $e_{(\mathrm{id}_{[i]}, \mathrm{id}_{[i]})}$.

    More precisely, let $\theta : \bigoplus_{(B_i \times B_i)_{C_2}} \mathbb{Z}/4 \to \mathbb{Z}/2$
    be the map that extracts the coefficient of $e_{(\mathrm{id}_{[i]},
      \mathrm{id}_{[i]})}$ and reduces it mod $2$. Then $\theta
    \!\mid_{\mathrm{ker}(\partial_i)} : \mathrm{ker}(\partial_i) \to \mathbb{Z}/2$
    exhibits the above quotient.
  \end{lemma}
  \begin{proof}
  Note that $e_{(\mathrm{id}_{[i]}, \mathrm{id}_{[i]})} \in \mathrm{ker}(\partial_i)$, so
  the restriction of $\theta$ to $\mathrm{ker}(\partial_i)$ remains surjective.
We need to show that
\[\mathrm{im}(\partial_{i+1}) + \mathrm{ker}(\partial_i) \cap \bigoplus_{B^o_i} \mathbb{Z}/4 + 2\mathrm{ker}(\partial_i)=
  \mathrm{ker}(\theta \!\mid_{\mathrm{ker}(\partial_i)}) \mathrm{.}\]
It is straightforward to check that the left hand side
is contained in the right hand side. It remains to show that every element
of the right hand side is contained in the left hand side.

First consider $i = 0$. Then $\partial_i$ is the zero-map and
$\mathrm{ker}(\partial_i) = \bigoplus_{(B_0 \times B_0)_{C_2}} \mathbb{Z}/4 \cong
\mathbb{Z}/4$. We see that $\mathrm{ker}(\theta\!\mid_{\mathrm{ker}(\partial_i)}) =
\mathrm{ker}(\theta) = 2\mathbb{Z}/4$ and $2 \mathrm{ker}(\partial_i) =
2\mathbb{Z}/4$ so the right hand side is contained in the left hand side as desired.

Otherwise we have $i \ge 1$. Let $x \in
\mathrm{ker}(\theta \!\mid_{\mathrm{ker}(\partial_i)})$ i.e.\ $x$ is an element
of $\mathrm{ker}(\partial_i)$ with an even coefficient of $e_{(\mathrm{id}_{[i]}, \mathrm{id}_{[i]})}$. We want to find $y \in
\bigoplus_{(B_{i+1} \times B_{i+1})_{C_2}} \mathbb{Z}/4$ such that the
diagonal part of $\partial_{i+1}(y)$ agrees with that of $x$, then $x -
\partial_{i+1}(y) \in \mathrm{ker}(\partial_i) \cap \bigoplus_{B^o_i}
\mathbb{Z}/4$ so $x \in \text{im}(\partial_{i+1}) +
\text{ker}(\partial_i)\cap\bigoplus_{B_i^o} \mathbb{Z}/4$ and we're done.

Write $x = x^d + x^o$ where $x^d \in \bigoplus_{B^d_i} \mathbb{Z}/4$ is the diagonal part and $x^o \in
\bigoplus_{B^o_i} \mathbb{Z}/4$ is the off-diagonal part. Consider what happens if we reduce everything
modulo $2$, which we denote with an overline. We get an element
\[\overline{x} = \overline{x}^d + \overline{x}^o \in \bigoplus_{B^d_i} \mathbb{Z}/2 \oplus
  \bigoplus_{B^o_i} \mathbb{Z}/2 \cong \bigoplus_{(B_i \times B_i)_{C_2}} \mathbb{Z}/2\mathrm{.}\]
Note the coefficient of $\overline{e}_{(\mathrm{id}_{[i]}, \mathrm{id}_{[i]})}$ in $\overline{x}$ is zero. Explicitly considering the face maps in $\bigoplus_{(B_\bullet \times
  B_\bullet)_{C_2}} \mathbb{Z}/4$ shows that the on-diagonal part of
$\partial_i(x^o)$ is a multiple of $2$, and we know
$\partial_i(x^o)+\partial_i(x^d) = \partial_i(x) = 0$, so $\partial_i(x^d)$ must
also be zero modulo $2$. That is, $\overline{x}^d$ is in the
kernel of the Moore boundary map for the simplicial abelian group
$\bigoplus_{B^d_\bullet} \mathbb{Z}/2 \cong \sigma M^0_\bullet / 2 \sigma M^0_\bullet$. However
the Dold-Kan isomorphism $\sigma$ commutes with reducing mod $2$, so this is
the simplicial abelian group $\sigma (M^0_\ast / 2M^0_\ast)$, where we observe
$M^0_\ast / 2M^0_\ast$ is the chain complex with $\mathbb{Z}/2$ in every
degree $\ge n$ and zero boundary maps. $M^0_\ast / 2M^0_\ast$ has homology
$\mathbb{Z}/2$ in every degree $\ge n$, so by Dold-Kan the $i$th homology group of the Moore complex
of $\bigoplus_{B^d_\bullet} \mathbb{Z}/2$ is also $\mathbb{Z}/2$. Moreover the
quotient defining this homology group
\[\mathrm{ker}\bigg(\partial_i : \bigoplus_{B^d_i} \mathbb{Z}/2 \to
\bigoplus_{B^d_{i-1}} \mathbb{Z}/2\bigg)\bigg/\mathrm{im}\bigg(\partial_{i+1} :
\bigoplus_{B^d_{i+1}} \mathbb{Z}/2 \to
\bigoplus_{B^d_i} \mathbb{Z}/2\bigg) \cong \mathbb{Z}/2\]
is explicitly given by extracting the
coefficient of $\overline{e}_{(\mathrm{id}_{[i]}, \mathrm{id}_{[i]})}$. Hence $\overline{x}^d$ represents zero
in homology, so there exists $\overline{y} \in \bigoplus_{B^d_{i+1}} \mathbb{Z}/2$ such that
$\partial_{i+1}(\overline{y}) = \overline{x}$.

Pick $y'$ a lift of $\overline{y}$ to $\bigoplus_{B^d_{i+1}} \mathbb{Z}/4$.
Including $\bigoplus_{B^d_{i+1}} \mathbb{Z}/4$ in
$\bigoplus_{(B_{i+1} \times B_{i+1})_{C_2}} \mathbb{Z}/4$, we can consider $y'$
an element of $\bigoplus_{(B_{i+1} \times B_{i+1})_{C_2}} \mathbb{Z}/4$. Then
$x^d$ and $\partial_{i+1}(y')$ agree modulo $2$, that is, $x^d -
\partial_{i+1}(y') = 2z$ for some $z \in \bigoplus_{B^d_i} \mathbb{Z}/4$.

Consider some
element of $B_i$. An element of
$B_i$ corresponds to a surjective map $[i] \twoheadrightarrow [j]$ for some $j$;
if this map sends $k \mapsto a_k$ then let us denote the corresponding element
of $B_i$ by the tuple
of values $a = (a_0, \dotsc, a_{i}=j)$. Let $a' = (a_0, a_0, a_1, a_2, \dotsc, a_{i}) \in B_{i+1}$ and $a'' =
(a_0, a_1, a_1, a_2, \dotsc, a_{i}) \in B_{i+1}$ (recall we are considering the
case where $i \ge 1$), and note that $a' \neq a''$.
Then $\partial_{i+1}(e_{(a', a'')}) = 2e_{(a, a)} + h$ where $h \in
\bigoplus_{B^o_i} \mathbb{Z}/4$ is
purely off-diagonal. By adding up elements of this form according to the
coefficients of $z$, we can find $y'' \in \bigoplus_{B^o_{i+1}} \mathbb{Z}/4$ such that
the on-diagonal part of $\partial_{i+1}(y'')$ equals $2z$.

Let $y = y' + y''$. Then $x^{d} - \partial_{i+1}(y)$ is purely off-diagonal, so $x - \partial_{i+1}(y) \in \mathrm{ker}(\partial_i)
\cap \bigoplus_{B^o_i} \mathbb{Z}/4$ as desired and the claim is proved.
\end{proof}

\begin{remark}
  We have used our computer program to double check this lemma for all values
  of $n$ and $i$ relevant for our later calculations (i.e.\ $0 \le n \le i \le
  6$). Indeed this was how we originally realised that the lemma might be
  true.
\end{remark}

\begin{proposition} \label{prop:r_on_diag_comp}
  Recall that $x$ represents a homology class $[x] \in
  \mathrm{ker}(\partial_i)/\mathrm{im}(\partial_{i+1})$, which by
  Corollary~\ref{cor:F_non_ab} we are identifying with the degree $i$ homotopy
  of the $n$-indexed diagonal component of $(\mathrm{THR}(\mathbb{Z}/2)^{\phi
    \mathbb{Z}/2})^{C_2}$. If the coefficient of $e_{(\mathrm{id}_{[i]},
    \mathrm{id}_{[i]})}$ in $x$ is even then $r_\ast([x]) = 0$. Otherwise
  $r_\ast([x])$ considered
  as an element of
\[\pi_i(\mathrm{THR}(\mathbb{Z}/2)^{\phi \mathbb{Z}/2}) \cong \bigoplus_{\substack{n',m,k
    \ge 0\\n'+m+k=i}} \mathbb{Z}/2\]
is the generator of the copy of $\mathbb{Z}/2$ indexed by $(n, 0, i-n)$, where
the above decomposition of $\pi_i(\mathrm{THR}(\mathbb{Z}/2)^{\phi
  \mathbb{Z}/2})$ is that of Proposition~\ref{prop:pi_thr_as_sum} (except using $n'$
as an index instead of $n$, since we are already using $n$).
\end{proposition}
\begin{proof}
By Remark~\ref{rem:suffices_residue_class} and
Lemma~\ref{lem:quotient_isomorphism} we deduce that for $x \in \mathrm{ker}(\partial_i)$, if the coefficient of
$e_{(\mathrm{id}_{[i]}, \mathrm{id}_{[i]})}$ in $x$ is odd then $r_\ast([x]) =
r_\ast([e_{(\mathrm{id}_{[i]}, \mathrm{id}_{[i]})}])$ and otherwise $r_\ast([x])
= 0$. Hence it only remains to compute $r_\ast([e_{(\mathrm{id}_{[i]}, \mathrm{id}_{[i]})}])$.

 Observe that the inclusion
  $\bigoplus_{B^d_i} \mathbb{Z}/4 \hookrightarrow \bigoplus_{(B_i \times B_i)_{C_2}} \mathbb{Z}/4$ takes
  the basis element $e_{\mathrm{id}_{[i]}}
  \in \bigoplus_{B_i} \mathbb{Z}/4 \cong \bigoplus_{B^d_i} \mathbb{Z}/4$
  to $e_{(\mathrm{id}_{[i]}, \mathrm{id}_{[i]})}$. This basis element $e_{\mathrm{id}_{[i]}}$ is
  in the kernel of the Moore differential on $\bigoplus_{B^d_\bullet} \mathbb{Z}/4 \cong \sigma
  M^0_\bullet$, and in fact represents the generator of $\pi_i(\abs{H\sigma
    M^0_\bullet}) \cong H_i(M^0_\ast) \cong \mathbb{Z}/4$.
  So $r_\ast([e_{(\mathrm{id}_{[i]}, \mathrm{id}_{[i]})}])$ is the image of $1
  \in \mathbb{Z}/4 \cong \pi_i\left(\abs{\bigoplus_{B^d_\bullet} H\mathbb{Z}/4}\right)$
  under the map on $\pi_i$ induced by
  \begin{equation} \label{eq:apply_to_id_i_map} \begin{aligned} &\Bigg\lvert\bigoplus_{B^d_\bullet} H\mathbb{Z}/4\Bigg\rvert \to \Bigg\lvert \bigoplus_{(B_\bullet \times
        B_\bullet)_{C_2}} H\mathbb{Z}/4 \Bigg\rvert \to
    \Bigg\lvert \bigoplus_{B_\bullet} H\mathbb{Z}/4^{\phi
          C_2} \Bigg\rvert\\
      &\xrightarrow{\cong} \Bigg\lvert H\mathbb{Z}/4^{\phi C_2} \otimes_{H\mathbb{Z}/4} H\sigma M_\bullet \Bigg\rvert \hookrightarrow
  \mathrm{THR}(\mathbb{Z}/4)^{\phi \mathbb{Z}/4}
\mathrm{,}\end{aligned} \end{equation}
  where as discussed in Lemma~\ref{lem:genuine_to_geometric_expand}
  the composite of the first two maps is the geometric realisation of the simplicial map
  \begin{equation*}%
    \bigoplus_{B^d_\bullet}
    H\mathbb{Z}/4 \cong \bigoplus_{B_\bullet} H\mathbb{Z}/4 \to
    \bigoplus_{B_\bullet} (H\mathbb{Z}/4)^{\phi C_2} \end{equation*}
  given component-wise by $H\mathbb{Z}/4 \to H\mathbb{Z}/4^{\phi C_2}$, and the final inclusion
  is as described in (\ref{eq:r_map_simp_def}).

Under the decomposition of $H\mathbb{Z}/4$-modules $H\mathbb{Z}/4^{\phi C_2}
\simeq \bigoplus_{m \ge 0} \Sigma^m H\mathbb{Z}/2$ we see
\begin{align*}{H\mathbb{Z}/4}^{\phi C_2} \otimes_{H\mathbb{Z}/4}
  H\sigma M_\bullet &\simeq \bigoplus_{m \ge 0} \Sigma^m H\mathbb{Z}/2 \otimes_{H\mathbb{Z}/4} H\sigma M_\bullet\\
                     &\cong \bigoplus_{m \ge 0} \Sigma^m H\mathbb{Z}/2 \otimes_{H\mathbb{Z}/4} H\sigma M^0_\bullet
\end{align*}
where we use the fact that $\sigma M_\bullet$ and $\sigma M^0_\bullet$ are
the same modulo $2$.
Now evaluating the geometric realisations we have
\[\Bigg\lvert \bigoplus_{B^d_\bullet} H\mathbb{Z}/4
  \Bigg\rvert \cong \Big\lvert H\sigma M^0_\bullet \Big\rvert \cong \bigoplus_{k \ge 0} \Sigma^{n+k} H\mathbb{Z}/4\]
and
\begin{align*}
  \Bigg\lvert \bigoplus_{m \ge 0} \Sigma^m H\mathbb{Z}/2 \otimes_{H\mathbb{Z}/4} H\sigma M^0_\bullet \Bigg\rvert
  &\cong \bigoplus_{m \ge 0} \Sigma^m H\mathbb{Z}/2 \otimes_{H\mathbb{Z}/4} \Big\lvert H\sigma M^0_\bullet \Big\rvert\\
  &\cong \bigoplus_{m \ge 0} \Sigma^m H\mathbb{Z}/2 \otimes_{H\mathbb{Z}/4} \bigoplus_{k \ge 0} \Sigma^{n+k} H\mathbb{Z}/4\\
  &\cong \bigoplus_{k \ge 0} \Sigma^{n+k} \bigoplus_{m \ge 0} \Sigma^m H\mathbb{Z}/2 \mathrm{.}
\end{align*}

Under these isomorphisms, the map (\ref{eq:apply_to_id_i_map}) becomes
\[\bigoplus_{k \ge 0} \Sigma^{n+k} H\mathbb{Z}/4 \to \bigoplus_{k \ge 0}
  \Sigma^{n+k} \bigoplus_{m \ge 0} \Sigma^m H\mathbb{Z}/2 \hookrightarrow
  \bigoplus_{n',m,k \ge 0} \Sigma^{n'+m+k} H\mathbb{Z}/2 \simeq
  \mathrm{THR}(\mathbb{Z}/2)^{\phi \mathbb{Z}/2}\]
where the first map is induced by $H\mathbb{Z}/4 \to H\mathbb{Z}/4^{\phi C_2}
\cong \bigoplus_{m \ge 0} \Sigma^m H\mathbb{Z}/4$, the inclusion includes the
$n'=n$ components, and the final equivalence is as described in Section~\ref{sec:computation_of_homotopy}. The map
$H\mathbb{Z}/4 \to \bigoplus_{m \ge 0} \Sigma^m H\mathbb{Z}/4$ induces a quotient
$\mathbb{Z}/4 \twoheadrightarrow \mathbb{Z}/2$ on $\pi_0$. Under
$\big\lvert\! \bigoplus_{B^d_\bullet} H\mathbb{Z}/4 \big\rvert \cong \bigoplus_{k \ge 0} \Sigma^{n+k}H\mathbb{Z}/4$ the homotopy group
$[e_{\mathrm{id}_{[i]}}]$ becomes the generator of $\pi_i(\Sigma^{n+k} H\mathbb{Z}/4)$ for $k =
i-n$. So its image under (\ref{eq:apply_to_id_i_map}) is the generator of
$\pi_i(\Sigma^{n'+m+k}H\mathbb{Z}/2)$ for $n' = n$, $m = 0$ and $k = i-n$.
\end{proof}

We are now ready to exploit the long exact sequence for the homotopy equaliser
of $r-f$.

\begin{theorem} \label{thm:tcr_ses_low_deg}
  We have short exact sequences
\begin{align*}
  0 \to \mathbb{Z}/2 \to &\pi_{-1}(\mathrm{TCR}(\mathbb{Z}/4)^{\phi \mathbb{Z}/2}) \to 0\\
  0 \to &\pi_{0}(\mathrm{TCR}(\mathbb{Z}/4)^{\phi \mathbb{Z}/2}) \to \mathbb{Z}/4 \to 0\\
  0 \to \mathbb{Z}/2 \to &\pi_{1}(\mathrm{TCR}(\mathbb{Z}/4)^{\phi \mathbb{Z}/2}) \to \mathbb{Z}/2 \to 0\\
  0 \to \mathbb{Z}/2 \to &\pi_{2}(\mathrm{TCR}(\mathbb{Z}/4)^{\phi \mathbb{Z}/2}) \to (\mathbb{Z}/2)^2 \to 0\\
  0 \to (\mathbb{Z}/2)^2 \to &\pi_{3}(\mathrm{TCR}(\mathbb{Z}/4)^{\phi \mathbb{Z}/2}) \to (\mathbb{Z}/2)^4 \to 0\\
  0 \to (\mathbb{Z}/2)^4 \to &\pi_{4}(\mathrm{TCR}(\mathbb{Z}/4)^{\phi \mathbb{Z}/2}) \to (\mathbb{Z}/2)^6 \oplus \mathbb{Z}/4 \to 0\\
  0 \to (\mathbb{Z}/2)^9 \to &\pi_{5}(\mathrm{TCR}(\mathbb{Z}/4)^{\phi \mathbb{Z}/2}) \to (\mathbb{Z}/2)^{10} \to 0 \mathrm{.}
\end{align*}
\end{theorem}
\begin{proof}
The
long exact sequence of homotopy groups
for the fibre sequence
\[\mathrm{TCR}(\mathbb{Z}/4)^{\phi \mathbb{Z}/2} \to
  (\mathrm{THR}(\mathbb{Z}/4)^{\phi \mathbb{Z}/2})^{C_2} \xrightarrow{r-f}
  \mathrm{THR}(\mathbb{Z}/4)^{\phi \mathbb{Z}/2}\]
splits into short exact sequences
\[0 \to \text{coker}((r-f)_{i+1}) \to \pi_i(\mathrm{TCR}(\mathbb{Z}/4)^{\phi
    \mathbb{Z}/2}) \to \text{ker}((r-f)_i) \to 0 \text{,}\]
so we just need to compute the kernels and cokernels of the maps $(r-f)_i$.

Using Propositions~\ref{prop:f_on_diag_comp} and \ref{prop:r_on_diag_comp} we extend our computer
program to compute the action of the maps $r$ and $f$ on homotopy groups.
See the code
repository for the full output; the computed kernels and cokernels are reproduced in the
following table:
\[\begin{tabular}{>{$}c<{$}>{$}c<{$}>{$}c<{$}>{$}c<{$}>{$}c<{$}}
  i & \pi_i((\mathrm{THR}(\mathbb{Z}/4)^{\phi \mathbb{Z}/2})^{C_2}) & \pi_i(\mathrm{THR}(\mathbb{Z}/4)^{\phi \mathbb{Z}/2}) & \mathrm{ker}((r-f)_i) & \mathrm{coker}((r-f)_i)\\ \hline
  0 & \mathbb{Z}/4 & \mathbb{Z}/2 & \mathbb{Z}/4 & \mathbb{Z}/2\\
  1 & (\mathbb{Z}/2)^4 & (\mathbb{Z}/2)^3 & \mathbb{Z}/2 & 0\\
  2 & (\mathbb{Z}/2)^7 & (\mathbb{Z}/2)^6 & (\mathbb{Z}/2)^2 & \mathbb{Z}/2\\
  3 & (\mathbb{Z}/2)^{13} & (\mathbb{Z}/2)^{10} & (\mathbb{Z}/2)^4 & \mathbb{Z}/2\\
  4 & (\mathbb{Z}/2)^{15} \oplus (\mathbb{Z}/4)^3 & (\mathbb{Z}/2)^{15} & (\mathbb{Z}/2)^6 \oplus \mathbb{Z}/4 & (\mathbb{Z}/2)^2\\
  5 & (\mathbb{Z}/2)^{27} & (\mathbb{Z}/2)^{21} & (\mathbb{Z}/2)^{10} & (\mathbb{Z}/2)^4\\
  6 & (\mathbb{Z}/2)^{34} & (\mathbb{Z}/2)^{28} & (\mathbb{Z}/2)^{15} & (\mathbb{Z}/2)^9 \text{.}
\end{tabular}\]

Let us describe the program output in detail in the case $i = 1$, since we will use this
in the next section. In Theorem~\ref{thm:thr_groups} we already computed
$\pi_1((\mathrm{THR}(\mathbb{Z}/4)^{\phi \mathbb{Z}/2})^{C_2}) \cong
(\mathbb{Z}/2)^4$ and $\pi_1(\mathrm{THR}(\mathbb{Z}/4)^{\phi \mathbb{Z}/2})
\cong (\mathbb{Z}/2)^3$. Let us denote the basis elements of
$\pi_1((\mathrm{THR}(\mathbb{Z}/4)^{\phi \mathbb{Z}/2})^{C_2})$ by $\{x_1, x_2,
x_3, x_4\}$, where $\{x_1, x_2\}$ generates the $0$-indexed diagonal component
isomorphic to $L_1^{(0)}F_{\mathbb{Z}/4}(\mathbb{Z}/2)$, $\{x_3\}$ generates the $1$-indexed
diagonal component $L_1^{(1)}F_{\mathbb{Z}/4}(\mathbb{Z}/2)$ and $\{x_4\}$ generates the off-diagonal component. Let the basis
elements of $\pi_1(\mathrm{THR}(\mathbb{Z}/4)^{\phi \mathbb{Z}/2})$ be denoted
by $\{y_1, y_2, y_3\}$, which corresponds to $(n, m, k) = (0, 0, 1), (0, 1, 0),
(1, 0, 0)$ respectively in the indexing
of Proposition~\ref{prop:pi_thr_as_sum}. Then the program computes that
$f$ takes $x_1 \mapsto 0$, $x_2 \mapsto y_1$, $x_3 \mapsto
0$ and $x_4 \mapsto y_2 + y_3$, and $r$ takes $x_1 \mapsto y_1$, $x_2 \mapsto 0$, $x_3
\mapsto y_3$ and $x_4 \mapsto 0$. From this we calculate that $(r-f)_1$ has kernel $\mathbb{Z}/2$ and trivial cokernel.
\end{proof}
\begin{remark}
  Note that all non-trivial homotopy groups of $\mathrm{TCR}(\mathbb{Z}/4)^{\phi
    \mathbb{Z}/2}$ are in degrees $\ge -1$. Indeed for $A$ a ring spectrum with involution such that $A$ and $A^{\phi
  \mathbb{Z}/2}$ are connective, then
  $(\mathrm{THR}(A)^{\phi \mathbb{Z}/2})^{C_2}$ and
  $\mathrm{THR}(A)^{\phi \mathbb{Z}/2}$ are also connective, so the long exact sequence
  shows that $\pi_i(\mathrm{TCR}(A)^{\phi
    \mathbb{Z}/2}) = 0$ for $i < -1$.
\end{remark}
\begin{corollary}
  The groups $\pi_i(\mathrm{TCR}(\mathbb{Z}/4)^{\phi \mathbb{Z}/2})$ are $4$-torsion
  for all $i$.
\end{corollary}
\begin{proof}
Since $\mathrm{TCR}(\mathbb{Z}/4)^{\phi \mathbb{Z}/2}$
is a ring spectrum, all the $\pi_i(\mathrm{TCR}(\mathbb{Z}/4)^{\phi
  \mathbb{Z}/2})$ are modules over the ring $\pi_0(\mathrm{TCR}(\mathbb{Z}/4)^{\phi
  \mathbb{Z}/2}) \cong \mathbb{Z}/4$ hence
must be $4$-torsion.
\end{proof}

\subsection{Resolving the extension problem in degree 1}
\label{sec:extension_i_1}

We would prefer to actually compute the homotopy groups, not just fit
them into exact sequences. The above shows $\pi_{-1}(\mathrm{TCR}(\mathbb{Z}/4)^{\phi
  \mathbb{Z}/2}) \cong \mathbb{Z}/2$ and $\pi_0(\mathrm{TCR}(\mathbb{Z}/4)^{\phi
  \mathbb{Z}/2}) \cong \mathbb{Z}/4$. The following analysis lets us additionally
compute $\pi_1$, and makes some progress in simplifying the problem for higher homotopy groups.

The key observation is that the map $r-f$ almost splits as a direct sum. Recall
how the maps $f$ and $r$ are defined in Corollary~\ref{cor:r_f_map_def} in terms of the decompositions of
$\mathrm{THR}(\mathbb{Z}/4)^{\phi \mathbb{Z}/2}$ and
$(\mathrm{THR}(\mathbb{Z}/4)^{\phi \mathbb{Z}/2})^{C_2}$ described in
Corollary~\ref{cor:decomp}. The map $f$ sends the diagonal summand indexed by
$n$ to the diagonal summand indexed by $n$, and sends the off-diagonal summand indexed
by $(n, m)$ into the off-diagonal summands indexed by $(n, m)$ and $(m, n)$. The
map $r$ is zero on the off-diagonal summands, and sends the diagonal summand indexed
by $n$ into the summands indexed $(n, m)$ for $m \ge 0$. So $r$ preserves
the $n$ index, and so does $f$ if we restrict to the diagonal summands. We
record this observation in the following lemma.

Let
\[D_n \coloneqq \left(\Sigma^{n \rho} H\mathbb{Z}/4 \otimes_{N_{\{e\}}^{C_2} (H\mathbb{Z}/4)}
    N_{\{e\}}^{C_2}(H\mathbb{Z}/2) \right)^{C_2}\]
denote the $n$-indexed diagonal summand of $(\mathrm{THR}(\mathbb{Z}/4)^{\phi
  \mathbb{Z}/2})^{C_2}$, and let
\[D \coloneqq \bigoplus_{n \ge 0} D_n \text{.}\]

\begin{lemma}The restriction $(r-f)|_D: D \to \mathrm{THR}(\mathbb{Z}/4)^{\phi \mathbb{Z}/2}$
is a direct sum of maps
\[\gamma_n : D_n \to \bigoplus_{m \ge 0} \Sigma^{n+m}H\mathbb{Z}/2 \otimes_{H \mathbb{Z}/4} H\mathbb{Z}/2 \mathrm{,}\]
where the right hand side
is the sum of the $(n, m)$-indexed summands of $\mathrm{THR}(\mathbb{Z}/4)^{\phi
\mathbb{Z}/2}$ for $m \ge 0$.
\end{lemma}
\begin{proof}
  Immediate from Corollary~\ref{cor:r_f_map_def}.
\end{proof}

Homotopy fibres preserve direct sums, so we have
\[\mathrm{hofib}((r-f)|_D) \cong \bigoplus_{n \ge 0} \mathrm{hofib}(\gamma_n) \mathrm{.}\]
Our computer program computes the action of $r-f$ on
homotopy in terms of the
decomposition of Corollary~\ref{cor:decomp}, so for each $n$ we can easily restrict
to obtain the action of the map $\gamma_n$, and from this calculate a long
exact sequence of homotopy groups for $\mathrm{hofib}(\gamma_n)$.

Let us record the resulting short exact sequences involving $\pi_1(\mathrm{hofib}(\gamma_n))$.

\begin{lemma} \label{lem:pi_1_x_n}
  We have
\begin{align*}
  0 \to (\mathbb{Z}/2)^2 \to &\pi_1(\mathrm{hofib}(\gamma_0)) \to \mathbb{Z}/2 \to 0\\
  0 \to \mathbb{Z}/2 \to &\pi_1(\mathrm{hofib}(\gamma_1)) \to 0
\end{align*}
and $\pi_1(\mathrm{hofib}(\gamma_n)) = 0$ for $n > 1$.
\end{lemma}
\begin{proof}
  From the long exact sequence of homotopy groups we get short exact sequences
  \[0 \to \text{coker}(\pi_{i+1}(\gamma_n)) \to \pi_i(\text{hofib}(\gamma_n))
    \to \text{ker}(\pi_i(\gamma_n)) \to 0 \text{.}\]

  Immediate from the output of the computer program, we can read off the relevant
  restrictions of $r-f$ and compute kernels and cokernels. For example, recall our
  description of the action of $f$ and $r$ on $\pi_1$ in the proof of
  Theorem~\ref{thm:tcr_ses_low_deg}. We have $\pi_1(D_0) \cong
  L_1^{(0)}F_{\mathbb{Z}/4}(\mathbb{Z}/2) \cong (\mathbb{Z}/2)^2$ corresponding
  to the subgroup generated by $\{x_1, x_2\}$ in our previous computation,
  and $\pi_1$ of the sum of the $(0, m)$-indexed
  components of $\mathrm{THR}(\mathbb{Z}/4)^{\phi \mathbb{Z}/2}$ is
  $(\mathbb{Z}/2)^2$ corresponding to the subgroup generated by $\{y_1, y_2\}$
  previously. The map $f$ takes $x_1 \mapsto 0$ and $x_2 \mapsto y_1$ respectively, and the map
  $r$ takes $x_1 \mapsto y_1$ and $x_2 \mapsto 0$. So $\gamma_0$ is given by $x_1
  \mapsto y_1$ and $x_2 \mapsto -y_1$, and we conclude that $\pi_1(\gamma_0)$
  has kernel $\mathbb{Z}/2$.

  The full set of kernels and cokernels needed is as follows:
  \begin{alignat*}{2}\mathrm{coker}(\pi_2(\gamma_0)) &\cong (\mathbb{Z}/2)^{2}\quad\quad &&\mathrm{ker}(\pi_1(\gamma_0)) \cong \mathbb{Z}/2\\
  \mathrm{coker}(\pi_2(\gamma_1)) &\cong \mathbb{Z}/2 \quad\quad &&\mathrm{ker}(\pi_1(\gamma_1)) \cong 0 \text{.}
  \end{alignat*}
  For $n > i$ the short exact
  sequences are evidently all zero.
\end{proof}

Now we need to relate the homotopy groups of the restricted homotopy fibre
$\mathrm{hofib}((r-f)|_D)$ to those of the
original homotopy fibre $\mathrm{hofib}(r-f) \simeq \mathrm{TCR}(\mathbb{Z}/4)^{\phi
\mathbb{Z}/2}$. Ultimately this will allow us to compute
$\pi_1(\mathrm{TCR}(\mathbb{Z}/4)^{\phi \mathbb{Z}/2})$.

We first note a small homological algebra lemma that we will need in the
subsequent proposition.

\begin{lemma} \label{lem:homological}
  Suppose we have a map of short exact sequences in an abelian category
  \[\begin{tikzcd}
    0 \ar[r] & A \ar[r] \ar[d] & B \ar[r]
    \ar[d] & X \ar[r] \ar[d, equals] & 0\\
    0 \ar[r] & C \ar[r] & D \ar[r] & X \ar[r] & 0\\
  \end{tikzcd}\]
where the rightmost vertical map is equality. Then the left hand square is a
pullback. If the map $A \to C$ is an epimorphism then the left hand square is also a pushout.
\end{lemma}
\begin{proof}
  First note that a square
  \[\begin{tikzcd}
L \ar[r, tail] \ar[d] & M \ar[d, two heads] \\
    N \ar[r] & P
  \end{tikzcd}\]
  with the top horizontal map a monomorphism and the right
  vertical map an epimorphism is a pushout iff it is a pullback. Indeed if it is
  a pushout then we have an exact sequence
  \[L \to M \oplus N \to P \to 0 \text{,}\]
  but since $L \to M$ is monic so is the map $L \to M \oplus N$ and so
  \[0 \to L \to M \oplus N \to P\]
  is exact and the square is a pullback. Dually if the square is a pullback then
  using $M \to P$ epic we see that the square is also a pushout.

  Now consider the diagram
  \[\begin{tikzcd}
    A \ar[r] \ar[d] & B \ar[d]\\
    C \ar[r] \ar[d] & D \ar[d]\\
    \ast \ar[r] & X \text{.}
  \end{tikzcd}\]
Since $X$ is the cokernel of $C \to D$, the bottom square is a
pushout. As $C \to D$ is monic and $D \to X$ is epic, the bottom
square is also a pullback. Similarly as $X$ is the cokernel of $A
\to B$ the outer square is a pushout and a pullback. So by the
pasting law for pullbacks the upper square is a pullback.

If $A \to C$ is epic then so is $B \to D$ by the $4$-lemma applied to the
original map of short exact sequences, so we
deduce the upper square is also a pushout.
\end{proof}

\begin{proposition} \label{prop:tcr_pushout_square}
  Let $(r-f)_i$ denote the map induced by $r-f$ on homotopy groups
  $\pi_{i}((\mathrm{THR}(\mathbb{Z}/4)^{\phi \mathbb{Z}/2})^{C_2}) \to
  \pi_{i}(\mathrm{THR}(\mathbb{Z}/4)^{\phi \mathbb{Z}/2})$. Then we have a pushout
  and pullback square
\[\begin{tikzcd}[column sep=small]
    \mathrm{coker}((r-f)_{i+1} |_{\pi_{i+1}(D)}) \ar[r] \ar[d] & \pi_i(\mathrm{hofib}((r-f)|_D))
    \ar[d] \\
    \mathrm{coker}((r-f)_{i+1}) \ar[r]  &
    \pi_i(\mathrm{TCR}(\mathbb{Z}/4)^{\phi \mathbb{Z}/2})
  \end{tikzcd}\]
where the horizontal arrows come from the long exact sequences for the homotopy
fibres of $(r-f)|_D$ and $r-f$, and the vertical maps are induced by the
inclusion of $D$ into $(\mathrm{THR}(\mathbb{Z}/4)^{\phi \mathbb{Z}/2})^{C_2}$.
\end{proposition}
\begin{proof}

We have a map of fibre sequences:
\[\begin{tikzcd}
    \mathrm{hofib}((r-f)|_D) \ar[r] \ar[d] & D \ar[r, "(r-f)|_D"] \ar[d] &
    \mathrm{THR}(\mathbb{Z}/4)^{\phi \mathbb{Z}/2} \ar[d, equal]\\
    \mathrm{TCR}(\mathbb{Z}/4)^{\phi \mathbb{Z}/2} \ar[r] &
    (\mathrm{THR}(\mathbb{Z}/4)^{\phi \mathbb{Z}/2})^{C_2} \ar[r, "r-f"] &
    \mathrm{THR}(\mathbb{Z}/4)^{\phi \mathbb{Z}/2} \mathrm{,}
  \end{tikzcd}\]
giving a corresponding map of long exact sequences in homotopy, and so maps of
the short exact sequences involving the homotopy groups of the homotopy fibres.
 Note that since
  $D$ is a summand of $(\mathrm{THR}(\mathbb{Z}/4)^{\phi \mathbb{Z}/2})^{C_2}$,
  $\pi_i(D)$ is a summand of $\pi_i((\mathrm{THR}(\mathbb{Z}/4)^{\phi
    \mathbb{Z}/2})^{C_2})$, and so the map induced by $(r-f)|_D$ on degree $i$ homotopy
  groups is $(r-f)_i |_{\pi_i(D)}$.
Then we have a map of short exact sequences
\[\begin{tikzcd}[column sep=small]
    0 \ar[r] & \mathrm{coker}((r-f)_{i+1} |_{\pi_{i+1}(D)}) \ar[r] \ar[d] & \pi_i(\mathrm{hofib}((r-f)|_D))
    \ar[r] \ar[d] & \mathrm{ker}((r-f)_i|_{\pi_i(D)}) \ar[r] \ar[d] & 0\\
    0 \ar[r] & \mathrm{coker}((r-f)_{i+1}) \ar[r] &
    \pi_i(\mathrm{TCR}(\mathbb{Z}/4)^{\phi \mathbb{Z}/2})
    \ar[r] & \mathrm{ker}((r-f)_i) \ar[r] & 0
  \end{tikzcd}\]
The map
$\mathrm{coker}((r-f)_{i+1} |_{\pi_{i+1}(D)}) \to \mathrm{coker}((r-f)_{i+1})$ is evidently
surjective. We claim that the map $\mathrm{ker}((r-f)_i|_{\pi_i(D)}) \to
\mathrm{ker}((r-f)_{i})$ is an isomorphism; that is, we claim that
$\mathrm{ker}((r-f)_{i})$ is contained within $\pi_i(D)$. Let $0 \le n < m$. If we post-compose
$(r-f)_i$ with the projection map  onto the off-diagonal component of $\mathrm{THR}(\mathbb{Z}/4)^{\phi
  \mathbb{Z}/2}$ indexed by $(n, m)$
\[\pi_i(\mathrm{THR}(\mathbb{Z}/4)^{\phi
  \mathbb{Z}/2}) \twoheadrightarrow \pi_i(\Sigma^{n+m} H\mathbb{Z}/2 \otimes_{H\mathbb{Z}/4}
H\mathbb{Z}/2)\]
then we get the projection map onto the off-diagonal component of $(\mathrm{THR}(\mathbb{Z}/4)^{\phi
  \mathbb{Z}/2})^{C_2}$ indexed by $(n, m)$
\[\pi_i((\mathrm{THR}(\mathbb{Z}/4)^{\phi
  \mathbb{Z}/2})^{C_2}) \twoheadrightarrow \pi_i(\Sigma^{n+m} H\mathbb{Z}/2 \otimes_{H\mathbb{Z}/4}
H\mathbb{Z}/2)\]
 (observe that this is what happens
if we post-compose $f_i$ with the projection, and on homotopy $r_i$ only hits non-diagonal
components indexed by $(n', 0)$ for some $n'$). So $\mathrm{ker}((r-f)_{i})$ is
entirely contained within the kernel of every projection onto an off-diagonal
component, i.e.\ is contained in the diagonal part of the decomposition. We
conclude the right hand vertical  map is an isomorphism.

Now we are done by Lemma~\ref{lem:homological}.
\end{proof}
\begin{remark}
Our earlier computer calculations suffice
to compute the left hand vertical map of the pushout square, so we've
reduced the problem of computing $\pi_i(\mathrm{TCR}(\mathbb{Z}/4)^{\phi
  \mathbb{Z}/2})$ to that of computing
$\pi_i(\mathrm{hofib}((r-f)|_D)) \cong \bigoplus^i_{n = 0}
\pi_i(\mathrm{hofib}(\gamma_n))$
and the pushout.
In fact as we will see in the following calculation, even if
while computing $\pi_i(\mathrm{hofib}(\gamma_n))$ we run into extension problems
that we are unable to solve, sometimes we can still compute the pushout.
\end{remark}

\begin{theorem}
  We have
  \[\pi_1(\mathrm{TCR}(\mathbb{Z}/4)^{\phi \mathbb{Z}/2}) \cong (\mathbb{Z}/2)^2 \mathrm{.}\]
\end{theorem}
\begin{proof}
We apply the above proposition to compute $\pi_1$. Taking the direct sum of the
$n=0$ and $n=1$ short exact sequences from Lemma~\ref{lem:pi_1_x_n}, we see that
\begin{align*}
  \mathrm{coker}((r-f)_{2}|_{\pi_{2}(D)}) &\cong (\mathbb{Z}/2)^2 \oplus \mathbb{Z}/2 \text{,}\\
  \pi_1(\mathrm{hofib}((r-f)|_D)) &\cong \pi_1(\mathrm{hofib}(\gamma_0)) \oplus \pi_1(\mathrm{hofib}(\gamma_1)) \cong
  \pi_1(\mathrm{hofib}(\gamma_0)) \oplus \mathbb{Z}/2 \text{,}\\
  \mathrm{ker}((r-f)_1|_{\pi_1(D)}) &\cong \mathbb{Z}/2 \oplus 0
  \text{.}\end{align*}
We saw earlier that $\mathrm{coker}((r-f)_2) \cong
\mathbb{Z}/2$, and explicitly examining the action of $r-f$ on $\pi_2$ shows
that the quotient
\[\mathrm{coker}((r-f)_{2}
  |_{\pi_{2}(D)}) \twoheadrightarrow \mathrm{coker}((r-f)_2)\]
is the projection of
$(\mathbb{Z}/2)^2 \oplus \mathbb{Z}/2$ onto the second summand. Now
Proposition~\ref{prop:tcr_pushout_square} gives us a pushout square
\[\begin{tikzcd}
  (\mathbb{Z}/2)^2 \oplus \mathbb{Z}/2 \ar[r, hookrightarrow] \ar[d, twoheadrightarrow] \arrow[dr, phantom,
  "\ulcorner", pos=0.95, xshift=4ex] &
  \pi_1(\text{hofib}(\gamma_0)) \ar[d, twoheadrightarrow]\\
\mathbb{Z}/2 \ar[r, hookrightarrow] & \pi_1(\mathrm{TCR}(\mathbb{Z}/4)^{\phi \mathbb{Z}/2}) \text{.}\end{tikzcd}\]
Considering this pushout we find that
$\pi_1(\mathrm{TCR}(\mathbb{Z}/4)^{\phi \mathbb{Z}/2})$ is the cokernel of the composition
\[(\mathbb{Z}/2)^2 \to \pi_1(\mathrm{hofib}(\gamma_0)) \hookrightarrow
  \pi_1(\mathrm{hofib}(\gamma_0)) \oplus \mathbb{Z}/2 \text{,}\]
where the first map is the map $(\mathbb{Z}/2)^2 \cong \text{coker}(\pi_2(\gamma_0)) \to \pi_1(\text{hofib}(\gamma_0))$ from the exact
sequence in Lemma~\ref{lem:pi_1_x_n}, and the second map is the inclusion of the
first factor. The exact sequence from Lemma~\ref{lem:pi_1_x_n} tells us that
$\pi_1(\mathrm{hofib}(\gamma_0))/(\mathbb{Z}/2)^2 \cong \mathbb{Z}/2$, so we find
\[\pi_1(\mathrm{TCR}(\mathbb{Z}/4)^{\phi \mathbb{Z}/2}) \cong
\pi_1(\mathrm{hofib}(\gamma_0))/(\mathbb{Z}/2)^2 \oplus \mathbb{Z}/2 \cong
(\mathbb{Z}/2)^2 \mathrm{.} \qedhere\]
\end{proof}

\begin{remark}
Unfortunately the same approach is not enough on its own to compute higher
homotopy groups of $\mathrm{TCR}(\mathbb{Z}/4)^{\phi \mathbb{Z}/2}$. For example
we run into an extension problem when computing $\pi_2$
of the homotopy kernel of $\gamma_1$, and
this time we would actually need to solve this extension problem
in order to compute $\pi_2(\mathrm{TCR}(\mathbb{Z}/4)^{\phi
  \mathbb{Z}/2})$. Nevertheless these techniques can potentially constrain the set of possible homotopy
groups.  For instance applying this analysis to $\pi_4(\mathrm{TCR}(\mathbb{Z}/4)^{\phi
  \mathbb{Z}/2})$ gives a second proof that this group is $4$-torsion (note we
already knew this since it is a $\pi_0(\mathrm{TCR}(\mathbb{Z}/4)^{\phi
  \mathbb{Z}/2})$-module).
\end{remark}

\section{Asymptotic growth estimates}\label{sec:asymptotic}

Our results so far show that the first few homotopy groups of
$\mathrm{TCR}(\mathbb{Z}/4)^{\phi \mathbb{Z}/2}$ do not obey a $4$-periodic
pattern. We would like to further prove that the homotopy groups grow arbitrarily large,
and so there is no point beyond which they start repeating with any period. To
show this we will make asymptotic estimates that show that the logarithm of
the size of the group $\pi_i(\mathrm{TCR}(\mathbb{Z}/4)^{\phi \mathbb{Z}/2})$
grows quadratically in $i$. In order to do so we almost completely
compute the groups $L^{(n)}_iF_{\mathbb{Z}/4}(\mathbb{Z}/2)$ for all $i$ and $n$; this is a
result that may be of independent interest.

Recall as in Theorem~\ref{thm:tcr_ses_low_deg} the long exact sequence
associated to the homotopy fibre of $r-f$ and the resulting short exact
sequences. If we can estimate the sizes of the homotopy groups involved in the
long exact sequence, then this gives
us constraints on the sizes of the kernels and cokernels appearing in the short
exact sequences, and so gives bounds on the size of
$\pi_i(\mathrm{TCR}(\mathbb{Z}/4)^{\phi \mathbb{Z}/2})$. In particular note we get
these estimates without needing any control over the maps in the long exact
sequence. We describe this
general technique in the following lemma.

\begin{lemma} \label{lem:size_est}
  Let
  \[A \xrightarrow{u} B \to C \to D \xrightarrow{v} E\]
  be an exact sequence of finite abelian groups. Then we have bounds
  \[\mathrm{max}\left(1, \frac{\abs{B}}{\abs{A}}\right) \cdot \mathrm{max}\left(1,
      \frac{\abs{D}}{\abs{E}}\right) \le \abs{C} \le \abs{B} \cdot \abs{D} \text{.}\]
\end{lemma}
\begin{proof}
  We have a short exact sequence
  \[0 \to \text{coker}(u) \to C \to \text{ker}(v) \to 0\]
  so
  \[\abs{C} = \abs{\text{coker}(u)} \cdot \abs{\text{ker}(v)} \text{.}\]
  And we have bounds
  \[\text{max}\left(1, \frac{\abs{B}}{\abs{A}}\right) \le \abs{\text{coker}(u)} \le \abs{B}\]
  and
  \[\text{max}\left(1, \frac{\abs{D}}{\abs{E}}\right) \le \abs{\text{ker}(v)}
    \le \abs{D} \text{,}\]
  where note that if $\abs{B} < \abs{A}$ we take the trivial lower bound $\abs{\text{coker}(u)} \ge 1$, and similarly if $\abs{D} < \abs{E}$ we take
  the trivial lower bound $\abs{\text{ker}(v)} \ge 1$.
\end{proof}

Applying this lemma to the long exact sequence for $\mathrm{TCR}(\mathbb{Z}/4)^{\phi \mathbb{Z}/2}$, appropriate bounds on the sizes of $\pi_j((\mathrm{THR}(\mathbb{Z}/4)^{\phi
\mathbb{Z}/2})^{C_2})$ and $\pi_j(\mathrm{THR}(\mathbb{Z}/4)^{\phi
\mathbb{Z}/2})$ will give us upper and lower bounds for the size of $\pi_i(\mathrm{TCR}(\mathbb{Z}/4)^{\phi \mathbb{Z}/2})$.

First recall that we entirely understand the non-equivariant homotopy groups of $\mathrm{THR}(\mathbb{Z}/4)^{\phi
  \mathbb{Z}/2}$.

\begin{proposition}\label{prop:thr_phi2}
  The group $\pi_i(\mathrm{THR}(\mathbb{Z}/4)^{\phi
        \mathbb{Z}/2})$ has order $2^{\frac{1}{2}(i+1)(i+2)}$.
\end{proposition}
\begin{proof}
  Recall from Proposition~\ref{prop:pi_thr_as_sum} that
  \[\pi_i(\mathrm{THR}(\mathbb{Z}/4)^{\phi \mathbb{Z}/2}) \cong \bigoplus_{\substack{n,m,k
    \ge 0\\n+m+k=i}} \mathbb{Z}/2 \mathrm{.}\]
The number of $\mathbb{Z}/2$-summands is the number of ways of writing $i$ as
the sum of three non-negative integers $i = n+m+k$. For each $0 \le n \le i$,
there are $i-n+1$ possible choices of non-negative integers $m, k$ such that $m+k = i-n$. So in
total we have
\[\sum_{n = 0}^i i-n+1 = \frac{1}{2}(i+1)(i+2)\]
summands.
\end{proof}

As before, analysing $\pi_i((\mathrm{THR}(\mathbb{Z}/4)^{\phi
  \mathbb{Z}/2})^{C_2})$ takes more work.

\begin{lemma}\label{lem:thr_phi2_c2_order_prelim}
  The group $\pi_i((\mathrm{THR}(\mathbb{Z}/4)^{\phi
        \mathbb{Z}/2})^{C_2})$ has order
      \[2^{\stackrel{\left\lfloor(i+1)^2/4\right\rfloor}{\phantom{.}}}
        \cdot \prod_{0 \le n \le i}\abs{L_i^{(n)}F_{\mathbb{Z}/4}(\mathbb{Z}/2)} \mathrm{.}\]
\end{lemma}
\begin{proof}
  Recall from Corollary~\ref{cor:decomp} that
  \begin{align*}
    (\mathrm{THR}(\mathbb{Z}/4)^{\phi \mathbb{Z}/2})^{C_2} \simeq &\bigoplus_{n \ge 0} \left( \Sigma^{n \rho}H\mathbb{Z}/4 \otimes_{N_{\{e\}}^{C_2}(H\mathbb{Z}/4)} N_{\{e\}}^{C_2}(H\mathbb{Z}/2) \right)^{C_2} \\
    & \oplus \bigoplus_{0 \le n < m} \Sigma^{n+m} H\mathbb{Z}/2 \otimes_{H\mathbb{Z}/4} H\mathbb{Z}/2 \mathrm{.}\end{align*}
  By Corollary~\ref{cor:F_non_ab} we know that the $i$th homotopy
  group of the $n$-indexed diagonal summand is computed using non-abelian
  derived functors as $L_i^{(n)}F_{\mathbb{Z}/4}(\mathbb{Z}/2)$, and by general properties of
  non-abelian derived functors these groups are zero for $n > i$. So the diagonal part has order
  \[\prod_{0 \le n \le i}\abs{L_i^{(n)}F_{\mathbb{Z}/4}(\mathbb{Z}/2)} \mathrm{.}\]
  As we observed in Proposition~\ref{prop:pi_thr_c2_as_sum} the homotopy of the
  off-diagonal terms is given by
\begin{equation*}\pi_i\left( \bigoplus_{0 \le n < m} \bigoplus_{k \ge 0} \Sigma^{n+m+k}
    H\mathbb{Z}/2 \right) \cong
  \bigoplus_{\substack{n,m,k \ge 0\\n < m\\n+m+k=i}} \mathbb{Z}/2 \mathrm{.}\end{equation*}
Let us count the number of
$\mathbb{Z}/2$ summands appearing here. For each $0 \le k \le i$, there are
$\floor{(i-k-1)/2}+1$ ways to pick non-negative integers $n, m$ with $n < m$ and
$n+m = i-k$, so the total number of summands is
\begin{align*}\sum_{k = 0}^i \floor{(i-k-1)/2}+1 &= \begin{cases}i(i+2)/4 & \text{if $i$ is even}\\
  (i+1)^2/4 & \text{if $i$ is odd.}\end{cases}\\
&= \left\lfloor \frac{(i+1)^2}{4} \right\rfloor
  \end{align*}
  Multiplying the orders of the diagonal and off-diagonal parts gives the result.
\end{proof}

\subsection{Non-abelian derived functor computations}
\label{sec:nadf}

To make further progress we need to study $L_i^{(n)}F_{\mathbb{Z}/4}(\mathbb{Z}/2)$. As a
warm-up, let us compute $L_i^{(n)}F_{\mathbb{Z}/4}(\mathbb{Z}/4)$; this will come in useful later.

\begin{lemma}\label{lem:derived_f_z4}
  We have
  \[L_i^{(n)}F_{\mathbb{Z}/4}(\mathbb{Z}/4) \cong \pi_i^{C_2}(\Sigma^{n\rho} H\mathbb{Z}/4)
    \cong \begin{cases}\mathbb{Z}/4 & \text{if $n$ is
    even and $i = 2n$}\\\mathbb{Z}/2 & \text{if $n \le i < 2n$, or if $n$ is odd
                                      and $i = 2n$}\\ 0 & \text{otherwise}\end{cases}\]
\end{lemma}
\begin{proof}
  By the same reasoning as Corollary~\ref{cor:F_non_ab}, Lemma~\ref{lem:expansion_to_em} gives
  \[L_i^{(n)}F_{\mathbb{Z}/4}(\mathbb{Z}/4) \cong \pi_i^{C_2}\left(\Sigma^{n\rho} H\mathbb{Z}/4
      \otimes_{N_{\{e\}}^{C_2}(H\mathbb{Z}/4)}
      N^{C_2}_{\{e\}}(H\mathbb{Z}/4)\right) \mathrm{.}\]
  But $\Sigma^{n\rho} H\mathbb{Z}/4
      \otimes_{N_{\{e\}}^{C_2}(H\mathbb{Z}/4)}
      N^{C_2}_{\{e\}}(H\mathbb{Z}/4) \cong \Sigma^{n\rho} H\mathbb{Z}/4$ so
  \[L_i^{(n)}F_{\mathbb{Z}/4}(\mathbb{Z}/4) \cong \pi_i^{C_2}\left(\Sigma^{n\rho} H\mathbb{Z}/4 \right) \mathrm{.}\]
  This is precisely the Bredon homology of $S^{n \rho}$ with coefficients in the
  Mackey functor of the abelian group $\mathbb{Z}/4$ with trivial involution.
  For $n$ even this is the homology of the chain complex
  \[\dotsb 0 \to \mathbb{Z}/4 \xrightarrow{0} \mathbb{Z}/4 \xrightarrow{2}
    \dotsb \mathbb{Z}/4 \xrightarrow{2} \mathbb{Z}/4 \xrightarrow{0}
    \mathbb{Z}/4 \xrightarrow{2} \mathbb{Z}/4 \to 0 \dotsb \]
  where the non-zero groups are in degrees $n$ to $2n$. For $n$ odd the
  chain complex looks like
  \[\dotsb 0 \to \mathbb{Z}/4 \xrightarrow{2} \mathbb{Z}/4 \xrightarrow{0}
    \dotsb \mathbb{Z}/4 \xrightarrow{2} \mathbb{Z}/4 \xrightarrow{0}
    \mathbb{Z}/4 \xrightarrow{2} \mathbb{Z}/4 \to 0 \dotsb \]
  with non-zero groups in the same degrees. The homology of these complexes
  gives the claimed result.
\end{proof}

Now we recall some of the theory of non-abelian derived functors from
\cite{dold_non-additive_1958} and \cite{simson_connected_1974}.

\begin{definition}[Cross effect, originally introduced in \cite{eilenberg_groups_1953}]
  Let $T : \mathrm{Mod}_R \to \mathrm{Ab}$ be a functor with $T(0) = 0$. The second
  cross-effect is the functor $T_2 : (\mathrm{Mod}_R)^2 \to \mathrm{Ab}$ defined
  on objects by
  \[T_2(A_1, A_2) \coloneqq \text{Im}\big(\text{id}_{T(A_1 \oplus T_2)} - T(i_{A_1}p_{A_1})
    - T(i_{A_2}p_{A_2})\big)\]
  where $i_{A_i}$, $p_{A_i}$ are the canonical injections and projections of $A_i$
  into and out of $A_1 \oplus A_2$. The functor $T_2$ is defined on morphisms in
  a natural way; see \cite{simson_connected_1974}~Section~4 for full details.

  We define higher cross effects $T_k : (\text{Mod}_R)^k \to \text{Ab}$ inductively via
  \[T_k(A_1, \dotsc, A_k) \coloneqq \big(T_{k-1}(A_1, \dotsc, A_{k-2}, {-})\big)_{2}(A_{k-1},
    A_k) \text{.}\]
\end{definition}

\begin{lemma}\label{lem:f_cross_effect}
  The functor $F_R(M) \coloneqq (M \otimes_R M)^{C_2}$ has second cross-effect
  \[A_1, A_2 \mapsto A_1 \otimes_R A_2 \text{,}\]
  and third cross-effect zero---accordingly we say
  that $F_R$ is quadratic.
\end{lemma}
\begin{proof}
  We compute
  \[\left( F_R \right)_2(A_1, A_2) \coloneqq \text{Im}\left(\text{id}_{F_R(A_1 \oplus A_2)} -
    F_R(i_{A_1}p_{A_1}) - F_R(i_{A_2}p_{A_2})\right) \cong A_1 \otimes_R A_2 \text{.}\]
For the third cross-effect we have
  \[\left(F_R  \right)_3(A_1, A_2, A_3) \coloneqq \left(\left( F_R \right)_2(A_1,
      {-})\right)_2(A_2, A_3) \text{.}\]
  But $\left( F_R \right)_2(A_1,
      {-}) = A_1 \otimes_R ({-})$ is additive, so has second cross-effect
      zero. Hence $F_R$ has third cross-effect zero.
\end{proof}

\begin{lemma}\label{lem:f_g_exact}
  Let $R$ be a commutative ring and let $M$ be an $R$-module. We have a long exact sequence
  \[\scalebox{0.75}{$\dotsb$}\! \to L^{(n)}_iG_R(M) \to L^{(n)}_iF_R(M) \to
    L^{(n+1)}_{i+1}F_R(M) \to L^{(n)}_{i-1} G_R(M) \to
    L^{(n)}_{i-1}F_R(M) \to \!\scalebox{0.75}{$\dotsb$}\]
  where $G_R(M) = M \otimes_R M$ is the tensor square functor.

  We have $L^{(n)}_iG_R(M) = 0$ for $i < 2n$, hence $L^{(n)}_iF_R(M) \cong
  L^{(n+1)}_{i+1}F_R(M)$ for $i < 2n$. Explicitly this isomorphism is
  given by the suspension map $\sigma_i^{(n)}(M)$ of
  \cite{dold_non-additive_1958} Equation~2.4.
\end{lemma}
\begin{proof}
  This follows by Corollary~3.9 of \cite{dold_non-additive_1958} and surrounding
  remarks, noting by the previous lemma that $\left( F_R
  \right)_2(A, A) = G_R(A)$ and $F_R$ is quadratic.
\end{proof}
\begin{remark}\label{rem:l_f_z2_exact1}
We will use this lemma in the case $R = \mathbb{Z}/4$ and $M = \mathbb{Z}/2$, where recall from Lemma~\ref{lem:l_g_comp} that
  \[L^{(n)}_iG_{\mathbb{Z}/4}(\mathbb{Z}/2) \cong \begin{cases}0 & \text{if $i < 2n$}\\ \mathbb{Z}/2 & \text{if $i \ge
                                                             2n$.}\end{cases} \]
\end{remark}

  For $i < 2n$ we have isomorphisms
  \[L^{(n)}_iF_R(M) \xrightarrow[\cong]{\sigma_i^{(n)}(M)}
  L^{(n+1)}_{i+1}F_R(M) \xrightarrow[\cong]{\sigma_{i+1}^{(n+1)}(M)} L^{(n+2)}_{i+2}F_R(M) \xrightarrow[\cong]{\sigma_{i+2}^{(n+2)}(M)}
  \dotsb \mathrm{.}\]
Accordingly we can study these stable values. In fact this stabilisation occurs
for the non-abelian derived functors of any functor (\cite{dold_non-additive_1958} Equation~3.11), so we can make
the following general definition.

\begin{definition}
  Let $X : \mathrm{Mod}_R \to \mathrm{Ab}$ be a functor with $X(0) = 0$. The left stable derived functors of $X$ are defined by
  \[L_q^sX(M) \coloneqq \mathrm{colim}\left(L_{q}^{(0)}X(M)
      \xrightarrow{\sigma_q^{(0)}(M)} L_{q+1}^{(1)}X(M)
      \xrightarrow{\sigma_{q+1}^{(1)}X(M)} L_{q+2}^{(2)}X(M) \to \dotsb \right) \text{.}\]
\end{definition}
\begin{remark}\label{rem:stable_nonab_i_less_zero}
  Since by Remark~\ref{rem:nonab_i_less_n} $L_i^{(n)}X(M) = 0$ for $i < n$, we see that $L_q^sX(M) = 0$ for $q < 0$.
\end{remark}
\begin{remark}\label{rem:ls_f_z4}
  By Lemma~\ref{lem:derived_f_z4} we have $L_{q+n}^{(n)}F_{\mathbb{Z}/4}(\mathbb{Z}/4) \cong
  \mathbb{Z}/2$ for $n > q \ge 0$.
  So
  \[L^s_{q}F_{\mathbb{Z}/4}(\mathbb{Z}/4) \cong \mathbb{Z}/2\]
  for $q \ge 0$.
\end{remark}

Simson and Tyc study stable derived functors in \cite{simson_connected_1974}. In
particular they compute the left stable derived functor of J. H. C. Whitehead's
functor $\Lambda_R: \mathrm{Mod}_R \to \mathrm{Mod}_R$ for certain rings $R$. This is
of interest to us since $\Lambda_R$ and $F_R$ are naturally isomorphic on free modules, so
have the same non-abelian derived functors. While Simson and Tyc
do not consider the case $R = \mathbb{Z}/4$, we can still make use of some of their
techniques.

\begin{proposition}\label{prop:ls_f_z2}
  For $q \ge 0$ we have
  \[L_q^sF_{\mathbb{Z}/4}(\mathbb{Z}/2) \cong (\mathbb{Z}/2)^{q+1} \mathrm{.}\]
\end{proposition}
\begin{proof}
  Applying \cite{simson_connected_1974} Corollary~5.8 to the short exact
  sequence
  \[0 \to \mathbb{Z}/2 \xrightarrow{2} \mathbb{Z}/4 \to \mathbb{Z}/2 \to 0\]
  gives a long exact sequence
  \[\dotsb \to L^s_qF_{\mathbb{Z}/4}(\mathbb{Z}/2) \to L^s_q(\mathbb{Z}/4) \to
    L^s_q(\mathbb{Z}/2) \to L^s_{q-1}F_{\mathbb{Z}/4}(\mathbb{Z}/2) \to \dotsb \mathrm{.}\]
  As noted in Remark~\ref{rem:ls_f_z4}, $L^s_qF_{\mathbb{Z}/4}(\mathbb{Z}/4) \cong
  \mathbb{Z}/2$ for $q \ge 0$. We claim that the maps
  $L^s_qF_{\mathbb{Z}/4}(\mathbb{Z}/2) \to L^s_qF_{\mathbb{Z}/4}(\mathbb{Z}/4)$ in the sequence, which come
  from applying $L^s_qF_{\mathbb{Z}/4}$ to
  $\mathbb{Z}/2 \xrightarrow{2} \mathbb{Z}/4$, are zero.

  Let us work through how to compute this map from the definitions. The map $\mathbb{Z}/2[n] \xrightarrow{2}
  \mathbb{Z}/4[n]$ corresponds to the map of free resolutions
  \[\begin{tikzcd}
      \cdots \ar[r, "2"] & \mathbb{Z}/4 \ar[d] \ar[r, "2"] & \mathbb{Z}/4 \ar[d]
      \ar[r, "2"] & \mathbb{Z}/4 \ar[r] \ar[d, "2"] & 0 \ar[d] \ar[r] & \cdots\\
      \cdots \ar[r] & 0 \ar[r] & 0 \ar[r] & \mathbb{Z}/4 \ar[r] & 0 \ar[r] & \cdots
  \end{tikzcd}\]
  where the non-zero vertical map is in degree $n$. To compute
  $L^{(n)}_iF_{\mathbb{Z}/4}(\mathbb{Z}/2) \to L^{(n)}_iF_{\mathbb{Z}/4}(\mathbb{Z}/4)$ we apply Dold-Kan to give
  a map of simplicial $\mathbb{Z}/4$-modules, apply $F_{\mathbb{Z}/4}$ level-wise giving a map
  of simplicial abelian groups, then take
  degree $i$ homotopy. Observe that after applying Dold-Kan, the level-wise maps
  are still all multiples of $2$. However applying $F_{\mathbb{Z}/4}$ to a map that is a
  multiple of $2$ gives the zero map (since $F_{\mathbb{Z}/4}(2\alpha) = 4F_{\mathbb{Z}/4}(\alpha) = 0$). So the resulting map of simplicial abelian
  groups is zero on the nose, and hence induces the zero map on homotopy.

  As a result, for $q \ge 0$ our long exact sequence splits up into short exact sequences of
  the form
  \[0 \to \mathbb{Z}/2 \to L^s_qF_{\mathbb{Z}/4}(\mathbb{Z}/2) \to L^s_{q-1}F_{\mathbb{Z}/4}(\mathbb{Z}/2)
    \to 0 \mathrm{.}\]
  By Remark~\ref{rem:stable_nonab_i_less_zero}, $L^s_{-1}F_{\mathbb{Z}/4}(\mathbb{Z}/2) =
  0$. And by \cite{simson_connected_1974} Corollary~8.10, $L^s_qF_{\mathbb{Z}/4}(M) \cong L^s_q
  \Lambda_{\mathbb{Z}/4}(M)$ is always $2$-torsion.
  So we see by induction on $q$ that
  \[L^s_qF_{\mathbb{Z}/4}(\mathbb{Z}/2) \cong (\mathbb{Z}/2)^{q+1} \mathrm{.} \qedhere\]
\end{proof}
\begin{remark}
  We can now compute the value of the stable derived functor
  $L^s_qF_{\mathbb{Z}/4}$ applied to any $\mathbb{Z}/4$-module. Indeed stable
  derived functors are additive, %
  and every $\mathbb{Z}/4$-module is a direct sum of copies of $\mathbb{Z}/4$ and
  $\mathbb{Z}/2$. So it suffices to compute the values of the functor at $\mathbb{Z}/4$ and
  $\mathbb{Z}/2$, which we did in Remark~\ref{rem:ls_f_z4} and
  Proposition~\ref{prop:ls_f_z2}.  %
\end{remark}

So we have computed $L_i^{(n)}F_{\mathbb{Z}/4}(\mathbb{Z}/2)$ in the stable range $i < 2n$. It
remains to see what happens when $i \ge 2n$. To do so we borrow some techniques from
Section~5 of \cite{dotto_geometric_2024}, but adapted to our current context of analysing
$\mathrm{TCR}(\mathbb{Z}/4)^{\phi \mathbb{Z}/2}$, whereas they were analysing
$\mathrm{TCR}(\mathbb{Z})^{\phi \mathbb{Z}/2}$.

\begin{lemma} \label{lem:sigma_Z4_C2_fibre_seq}
  We have a fibre sequence
  \[\begin{tikzcd}[row sep=small]\left( \Sigma^{n \rho + \sigma} H\mathbb{Z}/4
      \otimes_{N^{C_2}_{\{e\}}(H\mathbb{Z}/4)} N^{C_2}_{\{e\}}(H\mathbb{Z}/2)\right)^{C_2}
    \oplus \Sigma^{2n} \mathbb{Z}/2 \ar[d]\\
    (\Sigma^{n\rho}\mathbb{Z}/4)^{C_2} \ar[d]\\
    \left( \Sigma^{n \rho} H\mathbb{Z}/4
      \otimes_{N^{C_2}_{\{e\}}(H\mathbb{Z}/4)} N^{C_2}_{\{e\}}(H\mathbb{Z}/2) \right)^{C_2}
  \end{tikzcd}\]
  where $\sigma$ is the sign representation of $C_2$ (so $\rho = 1 + \sigma$).
\end{lemma}
\begin{proof}
  First apply Lemma~5.1 of \cite{dotto_geometric_2024} to the map $\mathbb{Z}/2
  \xrightarrow{2} \mathbb{Z}/4$. This map has cofibre $\mathbb{Z}/2$, so we get
  a square of $N^{C_2}_{\{e\}}(H\mathbb{Z}/4)$-modules
  \[\begin{tikzcd}[column sep=6em]
      \mathrm{Ind}^{C_2}_{\{e\}}(H\mathbb{Z}/2 \otimes_{\mathbb{S}} H\mathbb{Z}/2)
      \ar[r, "\mathrm{Ind}((H\mathbb{Z}/2) \otimes 2)"] \ar[d,
      "\widetilde{\mathrm{id}}"] & \mathrm{Ind}^{C_2}_{\{e\}}(H\mathbb{Z}/2
      \otimes_{\mathbb{S}} H\mathbb{Z}/4) \ar[d, "\widetilde{2 \otimes
        (H\mathbb{Z}/4)}"]\\
      N^{C_2}_{\{e\}}(H\mathbb{Z}/2) \ar[r, "{N(2)}" swap] & N^{C_2}_{\{e\}}(H\mathbb{Z}/4)
    \end{tikzcd}\]
  with total cofibre $N^{C_2}_{\{e\}}(H\mathbb{Z}/2)$. Here
  $\mathrm{Ind}^{C_2}_{\{e\}}$ is left adjoint to the forgetful functor from
  $N^{C_2}_{\{e\}}(H\mathbb{Z}/4)$-modules in $C_2$-spectra to $H\mathbb{Z}/4
  \otimes_{\mathbb{S}} H\mathbb{Z}/4$-modules in spectra, and the vertical maps
  are adjoint to the identity of $H\mathbb{Z}/4
  \otimes_{\mathbb{S}} H\mathbb{Z}/4$ and the map $2 \otimes (H\mathbb{Z}/4)$ respectively.

  Following the approach of \cite{dotto_geometric_2024} Lemma~5.2, we apply the
  functor
  \[H\mathbb{Z}/4 \otimes_{N^{C_2}_{\{e\}}(H\mathbb{Z}/4)} ({-})\mathrm{,}\]
  and note that this
  preserves the total cofibre. Simplifying the result somewhat, we get the square
  \[\begin{tikzcd}[column sep=5em]
      \mathrm{Ind}^{C_2}_{\{e\}}(H\mathbb{Z}/2 \otimes_{H\mathbb{Z}/4} H\mathbb{Z}/2)
      \ar[r, "\mathrm{Ind}((H\mathbb{Z}/2) \otimes 2)"] \ar[d,
      "\widetilde{\mathrm{id}}"] & \mathrm{Ind}^{C_2}_{\{e\}}(H\mathbb{Z}/2
      \otimes_{H\mathbb{Z}/4} H\mathbb{Z}/4) \ar[d, "\widetilde{2 \otimes
        (H\mathbb{Z}/4)}"]\\
      H\mathbb{Z}/4 \otimes_{N^{C_2}_{\{e\}}(H\mathbb{Z}/4)}
      N^{C_2}_{\{e\}}(H\mathbb{Z}/2) \ar[r, "{H\mathbb{Z}/4
        \otimes N(2)}" swap] & H\mathbb{Z}/4
      \otimes_{N^{C_2}_{\{e\}}(H\mathbb{Z}/4)} N^{C_2}_{\{e\}}(H\mathbb{Z}/4)
    \end{tikzcd}\]
  with total cofibre $H\mathbb{Z}/4 \otimes_{N^{C_2}_{\{e\}}(H\mathbb{Z}/4)} N^{C_2}_{\{e\}}(H\mathbb{Z}/2)$.

  The top horizontal map is zero, since $(H\mathbb{Z}/2) \otimes 2 : H\mathbb{Z}/2
  \otimes_{H \mathbb{Z}/4} H\mathbb{Z}/2 \to H\mathbb{Z}/2 \otimes_{H
    \mathbb{Z}/4} H\mathbb{Z}/4$ is evidently the zero map. The top
  right spectrum simplifies to $\mathrm{Ind}^{C_2}_{\{e\}}(H\mathbb{Z}/2)$, and
  the bottom right spectrum simplifies to $H\mathbb{Z}/4$. The left vertical
  map is given by the fold map
  \[H\mathbb{Z}/2 \otimes_{H\mathbb{Z}/4} H\mathbb{Z}/2 \oplus H\mathbb{Z}/2
    \otimes_{H\mathbb{Z}/4} H\mathbb{Z}/2 \to H\mathbb{Z}/2 \otimes_{H\mathbb{Z}/4} H\mathbb{Z}/2\]
  on underlying spectra, and has cofibre
  \[\Sigma^{\sigma} H\mathbb{Z}/4 \otimes_{N^{C_2}_{\{e\}}(H\mathbb{Z}/4)}
    N^{C_2}_{\{e\}}(H\mathbb{Z}/2)\]
  (in general the cofibre sequence
      $(C_2)_{+} \to \mathbb{S}^0 \to \mathbb{S}^{\sigma}$ gives us a sequence
      $\text{Ind}^{C_2}_{\{e\}}X \xrightarrow{\widetilde{\text{id}}} X \to \Sigma^{\sigma}X$). Hence we find the pushout of the top horizontal and
      left vertical maps is
      \[\left( \Sigma^{\sigma} H\mathbb{Z}/4 \otimes_{N^{C_2}_{\{e\}}(H\mathbb{Z}/4)}
      N^{C_2}_{\{e\}}(H\mathbb{Z}/2) \right) \oplus
      \mathrm{Ind}^{C_2}_{\{e\}}(H\mathbb{Z}/2) \mathrm{.}\]
  So the (co)fibre sequence defining the total cofibre is
  \[\begin{tikzcd}[row sep=small]
\Sigma^{\sigma} H\mathbb{Z}/4 \otimes_{N^{C_2}_{\{e\}}(H\mathbb{Z}/4)}
      N^{C_2}_{\{e\}}(H\mathbb{Z}/2) \oplus
      \mathrm{Ind}^{C_2}_{\{e\}}(H\mathbb{Z}/2) \ar[d]\\
      H\mathbb{Z}/4 \ar[d]\\
      H\mathbb{Z}/4 \otimes_{N^{C_2}_{\{e\}}(H\mathbb{Z}/4)}
      N^{C_2}_{\{e\}}(H\mathbb{Z}/2) \mathrm{.}
    \end{tikzcd}\]
  Taking the $\Sigma^{n\rho}$ suspension followed by genuine $C_2$-fixed points
  gives the desired fibre sequence.
\end{proof}
\begin{corollary}\label{cor:n_nplus_exact}
  We have a long exact sequence
  \begin{align*}\dotsb \to &L_{i+1}^{(n+1)}F_{\mathbb{Z}/4}(\mathbb{Z}/2) \oplus \pi_i(\Sigma^{2n} H\mathbb{Z}/2) \to
    \pi^{C_2}_i(\Sigma^{n\rho} H\mathbb{Z}/4) \to L_i^{(n)}F_{\mathbb{Z}/4}(\mathbb{Z}/2) \to\\
    &L_{i}^{(n+1)}F_{\mathbb{Z}/4}(\mathbb{Z}/2) \oplus \pi_{i-1}(\Sigma^{2n}H\mathbb{Z}/2) \to \pi_{i-1}^{C_2}(\Sigma^{n\rho}H\mathbb{Z}/4) \to L_{i-1}^{(n)}F_{\mathbb{Z}/4}(\mathbb{Z}/2) \to
    \dotsb \mathrm{.}\end{align*}
\end{corollary}
\begin{proof}
  Consider the long exact sequence in homotopy for the fibre sequence of Lemma~\ref{lem:sigma_Z4_C2_fibre_seq}.
  It remains to observe that for a $C_2$-spectrum $X$,
  \[\pi^{C_2}_i(\Sigma^{n\rho + \sigma}X) \cong \pi^{C_2}_{i+1}(\Sigma
    \Sigma^{n\rho + \sigma} X) \cong \pi^{C_2}_{i+1}(\Sigma^{(n+1)\rho}X)\]
  so
  \begin{align*}
    \pi^{C_2}_i\big(\Sigma^{n \rho + \sigma} H\mathbb{Z}/4 \otimes_{N^{C_2}_{\{e\}}(H\mathbb{Z}/4)} N^{C_2}_{\{e\}}(H\mathbb{Z}/2)\big) \cong L_{i+1}^{(n+1)}F_{\mathbb{Z}/4}(\mathbb{Z}/2) \mathrm{.}
  \end{align*}
  Hence identifying terms in the long exact sequence as appropriate
  non-abelian derived functors gives the result.
\end{proof}

We can now use this long exact sequence, together with our earlier results, to
almost completely compute $L_i^{(n)}F_{\mathbb{Z}/4}(\mathbb{Z}/2)$
(this appeared as Theorem~\ref{thm:intro_thm_nadf} in the introduction).

\begin{theorem}\label{thm:l_f_z2}
  For $i \ge 0$ fixed, $L_i^{(n)}F_{\mathbb{Z}/4}(\mathbb{Z}/2)$ takes the same value
  for all $0 \le n \le \floor{i/2}$. Moreover we determine this constant value, up to two possibilities in some cases:
  \begin{align*}L_i^{(0)} F_{\mathbb{Z}/4}(\mathbb{Z}/2) &\cong L_i^{(1)}
    F_{\mathbb{Z}/4}(\mathbb{Z}/2) \cong \dotsb \cong L_i^{\floor{i/2}}
    F_{\mathbb{Z}/4}(\mathbb{Z}/2)\\
    &\cong \begin{cases}
      (\mathbb{Z}/2)^{\floor{i/2}+1} \text{ or } (\mathbb{Z}/2)^{\floor{i/2}}
      \oplus \mathbb{Z}/4 \quad &\text{if $i \equiv 0 \text{ mod } 4$,}\\
      (\mathbb{Z}/2)^{\floor{i/2}+1} \text{ or } (\mathbb{Z}/2)^{\floor{i/2}+2} \quad &\text{if $i \equiv 1 \text{ mod } 4$,}\\
      (\mathbb{Z}/2)^{\floor{i/2}+1} \quad &\text{if $i \equiv 2 \text{ mod } 4$,}\\
      (\mathbb{Z}/2)^{\floor{i/2}+2} \quad &\text{if $i \equiv 3 \text{ mod } 4$.}
      \end{cases}\end{align*}
    For each $k$ we have
    $\abs{L_{4k}^{(0)}F_{\mathbb{Z}/4}(\mathbb{Z}/2)} =
    \abs{L_{4k+1}^{(0)}F_{\mathbb{Z}/4}(\mathbb{Z}/2)}$. That is, if we know
    the value for some $i$ congruent to $0$ mod $4$ then this tells us the value
    for $i+1$ and vice versa.

  For $\floor{i/2} < n \le i$ we have
  \[L_i^{(n)}F_{\mathbb{Z}/4}(\mathbb{Z}/2) \cong (\mathbb{Z}/2)^{i-n+1} \text{.}\]
  And 
  $L_i^{(n)}F_{\mathbb{Z}/4}(\mathbb{Z}/2) = 0$ for $n > i$ (recall from
  Remark~\ref{rem:nonab_i_less_n} that this is a general fact about non-abelian
  derived functors).
\end{theorem}
\begin{proof}
  The claim for $\floor{i/2} < n \le i$ follows from Proposition~\ref{prop:ls_f_z2} (see also
  Lemma~\ref{lem:f_g_exact}), since for $n \le i < 2n$ we have
  \[L_i^{(n)}F_{\mathbb{Z}/4}(\mathbb{Z}/2) \cong L^s_{i-n}F_{\mathbb{Z}/4}(\mathbb{Z}/2) \cong
    (\mathbb{Z}/2)^{i-n+1} \mathrm{.}\]

  The claim that $L_i^{(n)}F_{\mathbb{Z}/4}(\mathbb{Z}/2)$ takes a constant
  value for $0 \le n \le \floor{i/2}$ follows from the long exact sequence of
  Corollary~\ref{cor:n_nplus_exact}, since in Lemma~\ref{lem:derived_f_z4} we showed
  $\pi_i^{C_2}(\Sigma^{n\rho}H\mathbb{Z}/4) = 0$ for $i > 2n$, so for $i > 2n+1$
  we get exact sequences
  \[0 \to L_i^{(n)}F_{\mathbb{Z}/4}(\mathbb{Z}/2) \to L_i^{(n+1)}F_{\mathbb{Z}/4}(\mathbb{Z}/2) \to 0 \mathrm{.}\]

  To try to evaluate this constant value, we now analyse what happens near the boundary case $i = 2n$.

  Consider the exact sequence of Lemma~\ref{lem:f_g_exact} and
  Remark~\ref{rem:l_f_z2_exact1}, noting that
  \[L^{(n+1)}_{2n+1}F_{\mathbb{Z}/4}(\mathbb{Z}/2) \cong (\mathbb{Z}/2)^{n+1}
    \text{,}\]
  so we get an exact sequence
  \[\mathbb{Z}/2 \to L^{(n)}_{2n}F_{\mathbb{Z}/4}(\mathbb{Z}/2) \to
    (\mathbb{Z}/2)^{n+1} \to 0 \mathrm{.}\]
  This shows that
  \begin{equation}\label{eq:l_n_2n_prelim_options} L^{(n)}_{2n}F_{\mathbb{Z}/4}(\mathbb{Z}/2) \in \{(\mathbb{Z}/2)^{n+1},
    (\mathbb{Z}/2)^{n+2}, (\mathbb{Z}/2)^{n} \oplus \mathbb{Z}/4\} \mathrm{.}\end{equation}

  Next consider the following part of the sequence from
  Corollary~\ref{cor:n_nplus_exact}:
  \begin{equation} \label{eq:exact_seq_section_1} \begin{aligned} &L^{(n+1)}_{2n} F_{\mathbb{Z}/4}(\mathbb{Z}/2) \oplus
    \pi_{2n-1}(\Sigma^{2n}H\mathbb{Z}/2) \to \pi_{2n-1}^{C_2}(\Sigma^{n\rho}H\mathbb{Z}/4) \to L_{2n-1}^{(n)}F_{\mathbb{Z}/4}(\mathbb{Z}/2)\\ \to
    &L_{2n-1}^{(n+1)}F_{\mathbb{Z}/4}(\mathbb{Z}/2) \oplus \pi_{2n-2}(\Sigma^{2n}H\mathbb{Z}/2) \mathrm{.}\end{aligned} \end{equation}
  By Lemma~\ref{lem:derived_f_z4} we have $\pi_i^{C_2}(\Sigma^{n\rho}H\mathbb{Z}/4) \cong
  \mathbb{Z}/2$ for $n \le i < 2n$, so
  $\pi_{2n-1}^{C_2}(\Sigma^{n\rho}H\mathbb{Z}/4) \cong \mathbb{Z}/2$. Also
  substituting in $L_i^{(n)}F_{\mathbb{Z}/4}(\mathbb{Z}/2) \cong
  (\mathbb{Z}/2)^{i-n+1}$ for $n \le i < 2n$ and
  $\pi_{i}(\Sigma^{2n}H\mathbb{Z}/4) = 0$ for $i \ne 2n$, we get
  \[L^{(n+1)}_{2n} F_{\mathbb{Z}/4}(\mathbb{Z}/2) \to \mathbb{Z}/2 \to (\mathbb{Z}/2)^{n} \to
    (\mathbb{Z}/2)^{n-1} \text{.}\]
  But any map $(\mathbb{Z}/2)^{n} \to (\mathbb{Z}/2)^{n-1}$ has kernel of order
  at least $2$, so the map $\mathbb{Z}/2 \to (\mathbb{Z}/2)^{n}$ must be an
  injection, and hence the map $L^{(n+1)}_{2n} F_{\mathbb{Z}/4}(\mathbb{Z}/2) \to
  \mathbb{Z}/2$ is zero.

  Now we look at the part of the exact sequence of
  Corollary~\ref{cor:n_nplus_exact} leading up to the section in
  (\ref{eq:exact_seq_section_1}), and split the argument
  according to the parity of $n$. Let $n=2k+1$ be odd, then we get
  \begin{align*}0 \to L_{4k+3}^{(2k+1)}F_{\mathbb{Z}/4}(\mathbb{Z}/2) \to &L_{4k+3}^{(2k+2)}F_{\mathbb{Z}/4}(\mathbb{Z}/2)
    \oplus \mathbb{Z}/2 \to \mathbb{Z}/2 \to\\ &L_{4k+2}^{(2k+1)}F_{\mathbb{Z}/4}(\mathbb{Z}/2)
    \to L_{4k+2}^{(2k+2)}F_{\mathbb{Z}/4}(\mathbb{Z}/2) \to 0\end{align*}
  using that $\pi_i^{C_2}(\Sigma^{(2k+1)\rho}H\mathbb{Z}/4)$ is zero for $i >
  4k+2$ and isomorphic to $\mathbb{Z}/2$ at $i = 4k+2$ (by Lemma~\ref{lem:derived_f_z4} ), and as we just observed the map into that group
  for $i = 4k+1 = 2n-1$ is zero. Filling in the values of $L^{(n)}_iF_{\mathbb{Z}/4}(\mathbb{Z}/2)$
  that we have already computed gives
  \[0 \to L_{4k+3}^{(2k+1)}F_{\mathbb{Z}/4}(\mathbb{Z}/2) \to (\mathbb{Z}/2)^{2k+3}
    \to \mathbb{Z}/2 \to L_{4k+2}^{(2k+1)}F_{\mathbb{Z}/4}(\mathbb{Z}/2)
    \to (\mathbb{Z}/2)^{2k+1} \to 0 \mathrm{.}\]
  This shows $L^{(2k+1)}_{4k+2}F_{\mathbb{Z}/4}(\mathbb{Z}/2)$ has order at most $2^{2k+2}$, so
  considering the list of possibilities in (\ref{eq:l_n_2n_prelim_options}) we
  must have
  \[L^{(2k+1)}_{4k+2}F_{\mathbb{Z}/4}(\mathbb{Z}/2) \cong (\mathbb{Z}/2)^{2k+2} \mathrm{.}\]
  Then the rest of the exact sequence shows
  \[L^{(2k+1)}_{4k+3}F_{\mathbb{Z}/4}(\mathbb{Z}/2) \cong (\mathbb{Z}/2)^{2k+3} \mathrm{.}\]

  Finally consider the case where $n = 2k$ is even. We obtain a very similar exact sequence with the
  difference that now $\pi^{C_2}_{4k}(\Sigma^{2k\rho}H\mathbb{Z}/4) \cong
  \mathbb{Z}/4$, so we get
  \[0 \to L_{4k+1}^{(2k)}F_{\mathbb{Z}/4}(\mathbb{Z}/2) \to (\mathbb{Z}/2)^{2k+2}
    \to \mathbb{Z}/4 \to L_{4k}^{(2k)}F_{\mathbb{Z}/4}(\mathbb{Z}/2)
    \to (\mathbb{Z}/2)^{2k} \to 0\]
  There is no exact sequence $\mathbb{Z}/4 \to (\mathbb{Z}/2)^{2k+2} \to
  (\mathbb{Z}/2)^{2k} \to 0$, so considering our list of possibilities we must have
  \[L^{(2k)}_{4k}F_{\mathbb{Z}/4}(\mathbb{Z}/2) \in \{(\mathbb{Z}/2)^{2k+1},
    (\mathbb{Z}/2)^{2k} \oplus \mathbb{Z}/4\} \mathrm{.}\]
  If $L^{(2k)}_{4k}F_{\mathbb{Z}/4}(\mathbb{Z}/2) \cong (\mathbb{Z}/2)^{2k+1}$ then the map
  from $\mathbb{Z}/4$ has kernel $\mathbb{Z}/2$, so we find
  \[L^{(2k)}_{4k+1}F_{\mathbb{Z}/4}(\mathbb{Z}/2) \cong (\mathbb{Z}/2)^{2k+1} \mathrm{.}\]
  Otherwise $L^{(2k)}_{4k}F_{\mathbb{Z}/4}(\mathbb{Z}/2) \cong (\mathbb{Z}/2)^{2k}\oplus
  \mathbb{Z}/4$; in this case the map from $\mathbb{Z}/4$ is injective and we find
  \[L^{(2k)}_{4k+1}F_{\mathbb{Z}/4}(\mathbb{Z}/2) \cong (\mathbb{Z}/2)^{2k+2}
    \mathrm{.} \qedhere\]
\end{proof}

\begin{remark}
  Note this result agrees with the output of our computer program, which
  we have used to compute $L_i^{(n)}F_{\mathbb{Z}/4}$ for $i, n \le 6$. In the
  cases where the above theorem does not give a definite answer, the
  program gives us
  \begin{alignat*}{2} L_0^{(0)}F_{\mathbb{Z}/4}(\mathbb{Z}/2) &\cong \mathbb{Z}/4, \quad\quad
    &&L_1^{(0)}F_{\mathbb{Z}/4}(\mathbb{Z}/2) \cong (\mathbb{Z}/2)^2\text{,} \\
  L_4^{(0)}F_{\mathbb{Z}/4}(\mathbb{Z}/2) &\cong (\mathbb{Z}/2)^2 \oplus \mathbb{Z}/4, \quad\quad
    &&L_5^{(0)}F_{\mathbb{Z}/4}(\mathbb{Z}/2) \cong (\mathbb{Z}/2)^4 \text{.}\end{alignat*}
\end{remark}

\subsection{Application to size estimates}\label{sec:size_est}

Combining Theorem~\ref{thm:l_f_z2} with Lemma~\ref{lem:thr_phi2_c2_order_prelim} gives
us lower and upper bounds for $\abs{\pi_i((\mathrm{THR}(\mathbb{Z}/4)^{\phi \mathbb{Z}/2})^{C_2})}$.

\begin{theorem}\label{thm:thr_phi2_c2_order}
  We have bounds
  \[2^{(i+1)(5i+7)/8} \le \abs{\pi_i((\mathrm{THR}(\mathbb{Z}/4)^{\phi
          \mathbb{Z}/2})^{C_2})} \le 2^{(i+2)(5i+8)/8} \text{.}\]
\end{theorem}
\begin{proof}
   We work with the base 2 logarithms of the orders of the groups to make the
   notation easier. Recall Lemma~\ref{lem:thr_phi2_c2_order_prelim}, which says that
   \[\log_2 \, \abs{\pi_i((\mathrm{THR}(\mathbb{Z}/4)^{\phi
         \mathbb{Z}/2})^{C_2})} = \left\lfloor(i+1)^2/4\right\rfloor
        + \sum_{0 \le n \le
          i}\log_2\,\abs{L_i^{(n)}F_{\mathbb{Z}/4}(\mathbb{Z}/2)} \text{.}\]
  We apply the results of Theorem~\ref{thm:l_f_z2} to analyse the sum
  $\sum_{0 \le n \le i} \log_2\,\big|L_i^{(n)}F_{\mathbb{Z}/4}(\mathbb{Z}/2)\big|$.
  For the range $0 \le n \le \floor{i/2}$ we know %
  $L_i^{(n)}F_{\mathbb{Z}/4}(\mathbb{Z}/2)$ takes a constant value; we have
  \begin{align*}\sum_{0 \le n \le
      \floor{i/2}} \log_2\,\abs{L_{i}^{(n)}F_{\mathbb{Z}/4}(\mathbb{Z}/2)}
                                                     &= (\floor{i/2}+1)\log_2\,\abs{L_{i}^{(0)}F_{\mathbb{Z}/4}(\mathbb{Z}/2)}
  \end{align*}
  where $\log_2\,\abs{L_{i}^{(0)}F_{\mathbb{Z}/4}(\mathbb{Z}/2)} \in \{\floor{i/2}+1, \floor{i/2}+2\}$, giving lower and upper bounds for this
  term. For the range $\floor{i/2} < n \le i$ we have
  \begin{align*}\sum_{\floor{i/2} < n \le
      i} \log_2\,\abs{L_{i}^{(n)}F_{\mathbb{Z}/4}(\mathbb{Z}/2)} &= \sum_{\floor{i/2} < n \le i} (i-n+1)\\
                                                     &= (\ceil{i/2})(\ceil{i/2}+1)/2 \text{.}
  \end{align*}
   Combining
  these calculations and for convenience eliminating the floors and ceilings (at the cost of
  very slightly loosening the bounds) gives the result.
\end{proof}

Finally we have enough ingredients to use Lemma~\ref{lem:size_est} to get
asymptotic bounds on the order of the homotopy groups of
$\mathrm{TCR}(\mathbb{Z}/4)^{\phi \mathbb{Z}/2}$. The following appeared as
Theorem~\ref{thm:intro_asymptotics} in the introduction.

\begin{theorem}\label{thm:asymptotics}
  The groups
  \[\pi_i\left(\mathrm{TCR}(\mathbb{Z}/4)^{\phi \mathbb{Z}/2}\right)\]
  grow rapidly: the logarithm of the order asymptotically grows quadratically in $i$.

  More precisely, for $i \ge 0$ we have upper and lower bounds
  \[2^{(i-1)(i+1)/8} \le \abs{\pi_i\left(\mathrm{TCR}(\mathbb{Z}/4)^{\phi
          \mathbb{Z}/2}\right)} \le 2^{(i+2)(9i+20)/8} \text{.}\]
\end{theorem}
\begin{proof}
  Apply Lemma~\ref{lem:size_est} to the long exact sequence for
  $\mathrm{TCR}(\mathbb{Z}/4)^{\phi \mathbb{Z}/2}$, giving
\begin{equation*}\begin{aligned} \frac{\abs{\pi_i((\mathrm{THR}(\mathbb{Z}/4)^{\phi
        \mathbb{Z}/2})^{C_2})}}{\abs{\pi_i(\mathrm{THR}(\mathbb{Z}/4)^{\phi
        \mathbb{Z}/2})}} &\le \abs{\pi_i(\mathrm{TCR}(\mathbb{Z}/4)^{\phi \mathbb{Z}/2})}\\
                         & \le \abs{\pi_i((\mathrm{THR}(\mathbb{Z}/4)^{\phi
        \mathbb{Z}/2})^{C_2})} \cdot \abs{\pi_{i+1}(\mathrm{THR}(\mathbb{Z}/4)^{\phi
        \mathbb{Z}/2})} \mathrm{.}
\end{aligned}\end{equation*}

 Now by Proposition~\ref{prop:thr_phi2} computing
 $\abs{\pi_i(\mathrm{THR}(\mathbb{Z}/4)^{\phi \mathbb{Z}/2})}$ and
 Theorem~\ref{thm:thr_phi2_c2_order} proving bounds on
 $\abs{\pi_i((\mathrm{THR}(\mathbb{Z}/4)^{\phi \mathbb{Z}/2})^{C_2})}$, we find
  \begin{align*}
    \log_2\left( \abs{\pi_i(\mathrm{TCR}(\mathbb{Z}/4)^{\phi \mathbb{Z}/2})} \right) %
    &\ge (i+1)(5i+7)/8 - (i+1)(i+2)/2\\
    &= (i-1)(i+1)/8\end{align*}
  and
  \begin{align*}
    \log_2\left( \abs{\pi_i(\mathrm{TCR}(\mathbb{Z}/4)^{\phi \mathbb{Z}/2})} \right) %
    &\le (i+2)(5i+8)/8 + (i+2)(i+3)/2\\
    &= (i+2)(9i+20)/8\qedhere\end{align*}
\end{proof}

\newpage %
\phantomsection
\addcontentsline{toc}{section}{References}
\printbibliography

\end{document}